\documentclass[11pt, a4paper, article,reqno,]{amsart}
\usepackage{graphicx, xcolor,amsmath, amssymb}
\usepackage{fullpage}

\newtheorem{theorem}{Theorem}[section]

\newtheorem{definition}[theorem]{Definition}
\newtheorem{lemma}[theorem]{Lemma}

\theoremstyle{remark}
\newtheorem{remark}[theorem]{Remark}

\DeclareMathOperator{\dive}{div}

\numberwithin{equation}{section}

\title[High friction limit for Euler--Navier--Stokes--Korteweg]{High friction limit  for Euler--Korteweg and Navier--Stokes--Korteweg models via relative entropy approach} % arising in radiation hydrodynamics}

\author[G. Cianfarani Carnevale]{Giada Cianfarani Carnevale}
\address[Giada Cianfarani Carnevale]{Dipartimento di Ingegneria e Scienze dell'Informazione e Matematica, Universit\`a degli Studi dell'Aquila (Italy)}
\email{giada.cianfaranicarnevale@graduate.univaq.it}

\author[C. Lattanzio]{Corrado Lattanzio}
\address[Corrado Lattanzio]{Dipartimento di Ingegneria e Scienze dell'Informazione e Matematica, Universit\`a degli Studi dell'Aquila (Italy)}
\email{corrado@univaq.it}

\subjclass[2010]{}

\allowdisplaybreaks

\begin{document}

\begin{abstract}
The aim of this paper is to investigate the singular relaxation limits for the Euler--Korteweg and the Navier--Stokes--Korteweg system in the high friction regime. We shall prove that the viscosity term is present only in higher orders in the proposed scaling and therefore it does not affect the limiting dynamics, and the two models share the same equilibrium equation. The analysis of the limit is carried out using the relative entropy techniques in the framework of weak, finite energy solutions of the relaxation models converging toward smooth solutions of the equilibrium. The results proved here take advantage of the enlarged formulation of the models in terms of the \emph{drift velocity}  introduced in \cite{Bresh},  generalizing in this way the ones proved in \cite{LT2} for the Euler--Korteweg model.
\end{abstract}

%\tableofcontents

\maketitle

%% formato pagina
%\voffset-1cm
%\textheight22cm
%\textwidth15cm
%\oddsidemargin.4cm
%\evensidemargin.4cm

%\baselineskip15pt
%\thispagestyle{empty}

%\begin{quote}\footnotesize\baselineskip 12pt 
%{\sf Abstract.} 

%\vskip.15cm
%
%{\sf Keywords.}

%\vskip.15cm
%
%{\sf 2010 AMS subject classifications.} 
%76L05  % Shock waves and blast waves	
%(35L67,  35M30, 35Q31) %35L67 Shocks and singularities, 35M30 Systems of mixed type, 35Q31 Euler equations
%\end{quote}
\section{Introduction}
\label{sec:intro}
%The objective of this work is to study the emergence of gradient flows in Wassertain distance, like: $$\rho_t - \dive_x \left(\rho \nabla_x \frac{\delta\mathcal{E}}{\delta \rho} \right)=0$$ as high friction limit of an Euler flow generated by an energy functional:
The objective of this work is to study the high friction limit for  the Euler--Korteweg and the Navier--Stokes--Korteweg systems, that is:
\begin{equation}\label{ek}
	\left\{\begin{aligned}
		& \partial_t \rho + \dive m =0\\
		&  \partial_t m + \dive \left( \frac{m \otimes m}{\rho}\right) + \nabla p(\rho)  =
		 {2 \nu} \dive(\mu_L(\rho)Du) + \nu \nabla(\lambda_L(\rho) \dive u)
		\\
	& \qquad\qquad\qquad	+ \rho \nabla \left( k(\rho)\Delta\rho + \frac{1}{2}k'(\rho)|\nabla \rho|^2 \right) - \xi \rho u,
		\end{aligned}\right.
\end{equation}
where $t>0$, $x\in\mathbb{T}^n$, the $n$--dimensional torus, $\rho$ is the density, $m=\rho u $ is the momentum,  and the constants $\xi > 0$ and $\nu\geq 0$ stand for the large friction and the viscosity coefficient ($\nu=0$ for the  case of Euler--Korteweg system). 
% The  function $k(\rho)$ involved in the capillarity term  and the functions $\mu(\rho)$ and $\lambda(\rho)$ defining the viscous one are linked by the relations
% $\mu'(\rho) = \sqrt{\rho k(\rho)}$ and 
% $\lambda(\rho)= 2(\mu'(\rho)\rho - \mu(\rho))$.
As usual, in the viscosity terms of \eqref{ek}
$$Du= \frac{ \nabla u + {}^t\nabla u}{2}$$ 
is the symmetric part of the gradient $\nabla u$ and the Lam\'e coefficients $\mu_L(\rho)$ and $\lambda_L(\rho)$ verifies
\begin{equation}
    \label{lame}
    \mu_L(\rho)\geq 0;\ \frac2n \mu_L(\rho) +\lambda_L(\rho) \geq 0.
\end{equation}
Moreover, $p(\rho)$ stands for the pressure, connected to  the internal energy $e(\rho)$ by the relations 
\begin{equation}\label{eq:defh}
e'(\rho)= \frac{p(\rho)}{\rho^2};\ h(\rho)= \rho e(\rho);\ h''(\rho) = \frac{p'(\rho)}{\rho};\ p(\rho)=\rho h'(\rho) - h(\rho).
\end{equation}
As a consequence, we readily obtain 
\begin{equation*}
\rho \nabla (h'(\rho)) = \rho h''(\rho) \nabla \rho = \nabla p(\rho).
\end{equation*}
In what follows, and we shall confine ourselves to the case of monotone pressure, and, for simplicity we shall consider the classical $\gamma$--law 
 $p(\rho)= \rho^{\gamma}$ for  $\gamma > 1$, for which the function $h$ is given by
 \begin{equation*}
     h(\rho) = \frac{1}{\gamma-1}\rho^\gamma.
 \end{equation*}
 
The literature concerning these kind of systems, which include in particular Quantum Hydrodynamic models, is very wide and a complete description of it is beyond the main interest of our present research, which is focused in the study of the relaxation limit for weak, finite energy solutions of \eqref{ek}. In particular, we are not interested here in investigating the existence of such solutions, but solely in understanding their behavior in the high friction regime. 
% terms started with the papers \cite{Sagdeev,Zak,Gurevich}, 
 However, for some rigorous 
  mathematical studies  of such systems, regarding in particular the existence of weak solutions,  the dedicated reader may refer to \cite{AM1,AM2,AS,AS2} and the reference therein.

The high friction regime, after an appropriate time scaling, in both cases  is given by the following equation: 
\begin{equation}\label{eq:diff-limit}
\rho_t = \dive_x \left( \rho \nabla_x \left(h'(\rho) +  k(\rho)\Delta\rho + \frac{1}{2}k'(\rho)|\nabla \rho|^2 \right) \right),
\end{equation} 
as one can  easily check by performing the classical Hilbert expansion. Moreover, the rigorous study of this singular limit in terms of relative entropy techniques limit when $\nu>0$ does not present significant differences, and therefore
 we shall first discuss the case of Euler--Korteweg system in full details, and leave the discussion of the Navier--Stokes--Korteweg for the last section, where we shall emphasize only how to control the new terms due to the presence of the viscosity in \eqref{ek}. 
 Moreover, it is worth to observe here that, besides the natural condition \eqref{lame} needed to guarantee the dissipative nature of the viscosity terms, we shall assume here only appropriate uniform integrability conditions on that functions (which can be deduced from   a bound of their $L^1$ norm in terms of the energy), without a precise connection with the 
 capillarity coefficient $k(\rho)$, as  it is usually needed in the analysis of these models.
 
The kind of singular limits under investigation here enters in the realm of diffusive relaxations, for which hyperbolic systems of balance laws (as \eqref{ek} for $\nu=0$) converge in a diffusive scaling toward parabolic equilibrium systems. These kind of asymptotic analysis  has been addressed in various frameworks and with several techniques; in particular we refer to \cite{DM04} and the reference therein for the results concerning weak solutions and compactness arguments.
More recently, this kind of limits has been also successfully addressed by means of relative entropy techniques, starting from the well--known  case of 
 the Euler system with friction (obtained by choosing $k(\rho) =0$ in \eqref{ek} in addition to $\nu=0$) converging  to the porous media equation \cite{LT}. 
 It is worth recalling that,  as already pointed out before, this asymptotic behavior has been analyzed also before under many different viewpoints, and in particular for this remarkable example we refer to \cite{MM90,HMP05,HPW11}. However, the study of such limits with the present technique, even if it   is confined  to the case of smooth solutions  at equilibrium, has the advantage of obtaining a stability estimate and hence a rate of convergence as the relaxation parameter goes to zero.
 
 More recently, many other diffusive limits have been addressed following the same ideas; among others, see \cite{Bia19,FT19,HJT19,OR20,Carrillo}, and in particular here we recall the general framework introduced in \cite{GLT,LT2}, where the relative entropy calculation and the analysis of the diffusive limits have been presented in the general framework of abstract Euler flows generated by the first variation of an energy functional $\mathcal{E}(\rho)$:
\begin{equation*}
\left\{\begin{aligned}
		& \partial_t \rho + \dive (\rho u) =0\\
		&  \rho \partial_tu  + \rho u \cdot \nabla u =  - \rho \nabla \frac{\delta \mathcal{E}}{\delta \rho} - \xi \rho u.
		\end{aligned}\right.
\end{equation*}

The system $\eqref{ek}$ under consideration here belongs to this class of abstract flows for the following particular choice for $\mathcal{E}(\rho)$:
\begin{equation}\label{eq:potintro}
\mathcal{E}(\rho) =  \int \left (h(\rho) + \frac12 k(\rho) |\nabla \rho|^2 \right )dx.
\end{equation}
Referring in particular to the analysis of the large  limit, among other possible instances,  we recall  here that  in the paper  \cite{LT2} the Authors showed
 the emergence of the (Cahn--Hilliard type) equation \eqref{eq:diff-limit} as high friction limit of the Euler-Korteweg system solely in the case of constant capillarity $k(\rho) = C_k$. This result is based on the following general relative entropy relation for the aforementioned abstract Euler equations \cite{GLT,LT2} 
 \begin{equation*}
\begin{split}
& \frac{d}{dt} \left( \mathcal{E}(\rho|\bar{\rho}) + \int \frac{1}{2} \rho|u-\bar{u}|^2 \right) + \xi \int \rho|u-\bar{u}|^2 dx = \\ & \int \nabla \bar{u} : S(\rho|\bar{\rho}) dx - \int \rho \nabla \bar{u} : (u-\bar{u}) \otimes (u-\bar{u}) dx,
\end{split}
\end{equation*}
written here for  $(\rho,u)$ and $(\bar{\rho},\bar{u})$  smooth solutions of this system.  
The stress tensor $S$ appearing in the relation above can be defined in many examples of physical interest starting from the energy functional as follows:
\begin{equation*}
-\rho \nabla \frac{\delta \mathcal{E}}{\delta \rho} = \dive S.
\end{equation*}
In the particular case under consideration here, this relation becomes
\begin{equation*}
-\nabla p (\rho) + \rho \nabla \left( k(\rho)\Delta\rho + \frac{1}{2}k'(\rho)|\nabla \rho|^2 \right) = \dive S.
\end{equation*}
%$$\displaystyle{\mathcal{E}(\rho|\bar{\rho}) + \frac{1}{2} \rho(u-\bar{u})^2} = \eta(\rho,m|\bar{\rho},\bar{m})$$ is the relative entropy derived starting from the conservation of mass and the momentum of the systems. 
%
%
%where $\mathcal{E}(\rho)$ is the functional depending on the density that generates the evolution and plays the role of the usual stress tensor $S$. 
%
%The original derivation and all the related properties of $- \rho \nabla_x \frac{\delta \mathcal{E}}{\delta \rho}$ can be found in \cite{GLT} in a more general setting. 

The relation recalled above, and thus the corresponding control of the diffusive limit can be improved if we confine our attention to the specific form   of the  Euler-Korteweg systems 
\eqref{ek}, as it has been recently proved in \cite{Bresh}. Indeed, the relative entropy techniques turns out to be  more effective if one introduce the
\emph{drift velocity} 
\begin{equation*}
v= \frac{\nabla \mu(\rho)}{\rho},
\end{equation*}
where $\mu(\rho)$ satisfies $\mu'(\rho) = \sqrt{\rho k(\rho)}$. In this way, it is possible to obtain 
 an augmented formulation of \eqref{ek}, which, for the Euler-Korteweg system (that is, with $\nu=0$), reads  as follows:
\begin{equation}\label{ekb}
	\left\{\begin{aligned}
		& \partial_t \rho + \dive (\rho u) =0\\
		&  \partial_t( \rho u ) + \dive( \rho u\otimes u) + \nabla p( \rho)   = \dive(\mu(\rho)\nabla v) + \frac{1}{2}\nabla(\lambda(\rho)\dive v) - \xi \rho u\\
		& \partial_t (\rho v) + \dive(\rho v \otimes u) + \dive(\mu(\rho)^t\nabla u) + \frac{1}{2}\nabla (\lambda(\rho)\dive u)=0,
		\end{aligned}\right.
\end{equation}
 where $\lambda(\rho)= 2(\mu'(\rho)\rho - \mu(\rho))$  and,  thanks to the Bohm identity (see \cite{Bresh}),
 we also have
 \begin{equation*}
 \dive(\mu(\rho)\nabla v) + \frac{1}{2}\nabla(\lambda(\rho)\dive v) = \dive S_1
 \end{equation*}
 thus defining the new stress tensor  $S_1$ in $\eqref{ekb}$ due solely to the capillarity effects.
As we shall prove in the sequel, this approach will lead us to control the high friction limits for non constant capillarities,  obtaining the same advantages already pointed out in  \cite{Bresh} also in the context of diffusive relaxation, thus generalizing the results of \cite{LT2} for this particular system.
%Starting from $\eqref{ek}$, the stress tensor is defined in the following way: 
%$$ -\nabla p (\rho) + \rho \nabla \left( k(\rho)\Delta\rho + \frac{1}{2}k'(\rho)|\nabla \rho|^2 \right) = \dive S$$ 
%where $S$ is exactly $$S= \left( \rho \dive (k(\rho)\nabla \rho) + \frac{1}{2}(k(\rho)-k'(\rho)\rho)|\nabla\rho|^2 \right) \mathbb{I} -K(\rho)\nabla\rho\otimes\nabla\rho .$$ 
%Therefore in $\eqref{ekb}$ we have $\displaystyle{ \dive S = - \nabla p(\rho) + \dive K.}$
More precisely, the strategy is to  define a new \emph{momentum} $J=\rho v$ to then estimate the following relative entropy:
\begin{align*}
\eta(\rho,m,J | \bar{\rho},\bar{m},\bar{J}) & = \eta(\rho,m,J) - \eta(\bar{\rho},\bar{m},\bar{J}) - \bar{\eta}_{\rho}(\rho- \bar{\rho})- \bar{\eta}_m \cdot (m-\bar{m}) \\
& \  - \bar{\eta}_J \cdot(J-\bar{J}) \\
& = \frac{1}{2} \rho |u - \bar{u}|^2 + \frac{1}{2} \rho |v-\bar{v}|^2 + h(\rho|\bar{\rho}).
\end{align*}
In the present analysis, which involves a relaxation limit between two different diffusive theories,   the equilibrium (smooth) solution $(\bar{\rho},\bar{m},\bar{J})$ will solve the corresponding diffusive limiting equation, which shall then be recasted as an appropriate correction of the relaxing system \eqref{ekb}, as already done in previous works \cite{LT,LT2}.
%due to the transport equation for $J$ in $\eqref{ekb}$ the relative entropy  $\eta(\rho,m,J | \bar{\rho},\bar{m},\bar{J})$ reads as: $$\eta(\rho,m,J|\bar{\rho},\bar{m},J) = \eta(\rho,m,J) - \eta(\bar{\rho},\bar{m},\bar{J}) - \bar{\eta}_{\rho}(\rho- \bar{\rho})- \bar{\eta}_m (m-\bar{m}) - \bar{\eta}_J (J-\bar{J}),$$
%
%and using the new variables of $\eqref{ekb}$ we get explicitely:
%\begin{equation*}
%\begin{split}
%& \partial_t \eta( \rho, u,v| \bar{\rho}, \bar{u}, \bar{v}) = \partial_t \left( \frac{1}{2} \rho |u - \bar{u}|^2 + \frac{1}{2} \rho |v-\bar{v}|^2 + h(\rho|\bar{\rho}) \right) = \\ & - \frac{1}{\epsilon} \rho \nabla \bar{u} : (u-\bar{u}) \otimes (u-\bar{u}) - \frac{1}{\epsilon^2} \rho (u-\bar{u})^2 - \frac{\rho}{\bar{\rho}} \bar{e}(u-\bar{u}) - \frac{1}{\epsilon}p(\rho| \bar{\rho}) \dive \bar{u} \\ & - \frac{1}{\epsilon} \rho \; ^t\nabla \bar{u} : (v-\bar{v}) \otimes (v-\bar{v})  - \frac{1}{\epsilon}\rho( \mu''(\rho) \nabla \rho - \mu''(\bar{\rho}) \nabla \bar{\rho})) \cdot ((v - \bar{v}) \dive \bar{u} - (u - \bar{u}) \dive \bar{v}) \\ & - \frac{1}{\epsilon} \rho(\mu'(\rho) - \mu'(\bar{\rho}))((v-\bar{v})) \cdot \nabla (\dive\bar{u}) - (u - \bar{u}) \cdot \nabla (\dive \bar{v})).
%\end{split}
%\end{equation*}
%where $(\rho,u,v)$ is weak solution of $\eqref{ekb}$ and $(\bar{\rho},\bar{u},\bar{v})$ is the strong solution of the correction system for the limit equation. Once we have assumed some integrability condition on $\rho,u,v$ we are able to control this quantity and show a stability result thanks to a Gronwall estimate.

The outline of this work is as follows. In Section \ref{sec:hilbert}, after the appropriate time scaling,  we perform the Hilbert expansion of \eqref{ekb} in order to recognize the limit equation. Then  we rewrite the latter as a correction of the relaxation system \eqref{ekb} to take full advantage of the relative entropy tools. Section \ref{sec:relenes} is devoted to obtaining the relative entropy inequality, which will be used as an yardstick to measure the distance between the two solutions in the relaxation limit of the subsequent section. Finally, in Section \ref{sec:NSK} we describe our all results can be adapted in a straightforward way to the case of the Navier--Stokes--Korteweg model \eqref{ek} for $\nu>0$.
% in the Appendix we present the (formal) calculation in the case  both  solutions, including the relaxing one, are regular, to emphasize  the relative entropy fluxes involved in it.
%\section{Diffusive limit and Hilbert expansion for Euler-Kortweg system with friction}
\section{Hilbert expansion and formal diffusive limit for the Euler--Korteweg model}
\label{sec:hilbert}
In this section we shall present the correct  scaling for  which \eqref{ek}, and hence \eqref{ekb},  exhibits the desired  diffuse limit. More precisely, for $\xi = 1/\epsilon$, 
we rescale the time so that  $\partial_t \rightarrow \epsilon\partial_t$ and \eqref{ek} becomes:
\begin{equation}\label{ek-scaled}
	\left\{\begin{aligned}
		&  \partial_t \rho + \frac{1}{\epsilon} \dive m =0 \\
		&  \partial_tm + \frac{1}{\epsilon} \dive \left( \frac{m \otimes m}{\rho}\right) + \frac{1}{\epsilon} \nabla p( \rho)   
		= \frac{1}{\epsilon}\dive S_1 - \frac{1}{\epsilon^2} \rho u. 
%		& \partial_t (J) + \frac{1}{\epsilon} \dive \left( \frac{  J \otimes m}{\rho}   \right) + \dive S_2=0,
	\end{aligned}\right.
\end{equation}
Accordingly, \eqref{ekb} reads
\begin{equation}\label{ekb-scaled}
	\left\{\begin{aligned}
		& \partial_t \rho + \frac{1}{\epsilon} \dive (m) =0\\
		&  \partial_t( m) + \frac{1}{\epsilon} \dive \left( \frac{m \otimes m}{\rho}\right) + \frac{1}{\epsilon} \nabla p( \rho)   = \frac{1}{\epsilon}\dive S_1 - \frac{1}{\epsilon^2} \rho u\\
		& \partial_t (J) + \frac{1}{\epsilon} \dive \left( \frac{  J \otimes m}{\rho}   \right) + \dive S_2=0,
		\end{aligned}\right.
\end{equation}
where $J=\rho v$ and   (see \cite{Bresh} for further details)
\begin{equation*}
\dive  S_2 = \dive(\mu(\rho)^t\nabla u) + \frac{1}{2}\nabla (\lambda(\rho)\dive u).
\end{equation*}
 In order to perform the Hilbert expansion, we need to 
 introduce the asymptotic expansions of $\rho$ and $m$ in \eqref{ekb-scaled}, and the one for $J$ will follow, being $J = \rho v = \nabla \mu(\rho)$.
 To this end,
\begin{align*}
 &\rho = \rho_0 + \epsilon\rho_1 + \epsilon^2 \rho_2 + \cdots\\
  &m = m_{0} + \epsilon m_{1} + \epsilon^2 m_{2} + \cdots
  \end{align*}
 and collect the terms of the same order. From the mass conservation we get:
\begin{align*}
&O(\epsilon^{-1}): &  & \dive m_{0} = 0; \\ 
&O(1): &  &\partial_t \rho_0 + \dive m_{1} = 0; \\
&O(\epsilon): %& \partial_t \rho_1 + \dive m_{2} =0;\\
&  & \dots \\
\end{align*}
from the momentum equation we get:
\begin{align*}
& O(\epsilon^{-2}) :  &   &m_{0}=0; \\ 
& O(\epsilon^{-1}) :  &    &-m_{1} = \nabla p(\rho_0) - \dive S_1(\rho_0);& \\
 & O(1):  &    &\dots  \\
\end{align*}
%O(1): & -m_2 = \nabla( p'(\rho_0) \rho_1) - \dive_x K(\rho_0,m_1).
Hence, from these first relations,  we recover the equilibrium relation $m_{0} = 0$, the Darcy's law $m_{1} = - \nabla_x p(\rho_0) + \dive_x S_1(\rho_0)$, and 
  the following gradient flow dynamic for $\rho_0$:
\begin{equation}\label{gf}
\partial_t\rho_{0} + \dive \left( - \nabla p(\rho_0) + \dive S_1(\rho_0) \right)= 0,
\end{equation}
that is, the formal limit as $\epsilon \rightarrow 0$ of \eqref{ekb-scaled}.
%\rho_t + \dive_x \left( \rho \nabla_x \left(-h'(\rho) + K(\rho)\Delta \rho + \frac{1}{2}K'(\rho)|\nabla \rho|^2 \right) \right)= 0 ,where we have used the definition of the stress tensor $S$.

In order to compare weak solutions of  \eqref{ekb-scaled} and strong solutions of its parabolic equilibrium \eqref{gf} and take full advantage of the relative entropy estimate for hyperbolic systems, as  already done in \cite{LT, LT2}, we the latter  as Euler--Korteweg system with friction plus  an error term as follows.
Let us denote by $\bar{\rho}$ the (smooth)  solution of $\eqref{gf}$. Then $(\bar{\rho}, \bar{m} = \bar{\rho}\bar{u})$ solves
\begin{equation}\label{ek-scaledstrong}
	\left\{\begin{aligned}
		& \partial_t \bar{\rho} + \frac{1}{\epsilon} \dive \bar{m} =0\\
		&  \partial_t\bar{ m} + \frac{1}{\epsilon} \dive \left( \frac{\bar{m} \otimes \bar{m}}{\bar{\rho}}\right) + \frac{1}{\epsilon} \nabla p(\bar{ \rho})   = \frac{1}{\epsilon}\dive \bar{S_1} - \frac{1}{\epsilon^2} \bar{m} + e(\bar{\rho},\bar{m}),
%		& \partial_t (\bar{J}) + \frac{1}{\epsilon} \dive \left( \frac{ \bar{J} \otimes \bar{m}}{\bar{\rho}}   \right) + \frac{1}{\epsilon} \dive \bar{S_2}=0
		\end{aligned}\right.
\end{equation}
where
\begin{equation*}
\bar{m} =  \epsilon \left(- \nabla p(\bar{\rho})+   \dive S_1(\bar{\rho})\right).
\end{equation*}
Clearly, in \eqref{ek-scaledstrong}, the error term  $e(\bar{\rho},\bar{m})= \bar{e}$ is given by:
\begin{align}\label{eq:error}
 \bar{e} & = \frac{1}{\epsilon} \dive_x \left( \frac{\bar{m} \otimes \bar{m} }{\bar{\rho}} \right) + \bar{m}_t \nonumber\\ 
 & = \epsilon \dive_x \left( (- \nabla p(\bar{\rho})+   \dive S_1(\bar{\rho})) \otimes  (- \nabla p(\bar{\rho})+   \dive S_1(\bar{\rho}) \right) 
 \nonumber\\ 
 & \ +  \epsilon \left( - \nabla p(\bar{\rho})+   \dive S_1(\bar{\rho}) \right)_t 
 \nonumber\\ 
 & = O(\epsilon).
\end{align}
Introducing the notation $\bar{J}= \bar{\rho}\bar{v} = \nabla \mu(\bar{\rho})$, the equilibrium can be rewritten also as follows:
\begin{equation}\label{ekb-scaledstrong}
	\left\{\begin{aligned}
		& \partial_t \bar{\rho} + \frac{1}{\epsilon} \dive \bar{m} =0\\
		&  \partial_t\bar{ m} + \frac{1}{\epsilon} \dive \left( \frac{\bar{m} \otimes \bar{m}}{\bar{\rho}}\right) + \frac{1}{\epsilon} \nabla p(\bar{ \rho})   = \frac{1}{\epsilon}\dive \bar{S_1} - \frac{1}{\epsilon^2} \bar{m} + e(\bar{\rho},\bar{m})\\
		& \partial_t \bar{J} + \frac{1}{\epsilon} \dive \left( \frac{ \bar{J} \otimes \bar{m}}{\bar{\rho}}   \right) + \frac{1}{\epsilon} \dive \bar{S_2}=0.
		\end{aligned}\right.
\end{equation}
%for
%\begin{equation*}
%\bar{J}= \bar{\rho} \bar{ v} = \nabla \mu(\bar{\rho}).
%\end{equation*}
As already done previously \cite{LT,LT2}, in next section we as shall  validate  rigorously the large friction limit  using relative entropy estimates, but this time using the enlarged reformulation in terms of the drift velocity, thus considering the singular limit from \eqref{ekb-scaled} to \eqref{ekb-scaledstrong}.
\section{Relative entropy estimate for the Euler--Korteweg model}
\label{sec:relenes}
%The strategy will be to write a differential equation for the relative entropy using the equations of the system, then we will integrate in time and space under suitable boundary conditions to end up with a good estimate for the difference of the two solutions. \\ 
%As in the previous works (see \cite{GLT,LT,LT2}),
Let us start by  we start by recalling the  entropy--entropy flux pair $(\eta, Q)$ associated to the original Euler-Korteweg system \eqref{ek} with $\xi = 1/\epsilon$ and after the related time scaling.
Using the notation of \cite{GLT,LT2}, we obtain the potential energy (see \eqref{eq:potintro})
$$F(\rho, \nabla \rho)= h(\rho) + \frac{1}{2} k(\rho)|\nabla \rho|^2,$$ 
while  the kinetic energy reads
$$E_K = \frac{1}{2} \rho |u|^2.$$ 
Moreover, the couple $(\eta, Q)$ is defined in the following way:
\begin{align*}
\eta(\rho,m, \nabla \rho) &= \frac{1}{2} \rho|u|^2+ \frac{1}{2} k(\rho)|\nabla \rho|^2 + h(\rho);\\
Q(\rho,m, \nabla \rho) &= \frac{1}{2}\rho u |u|^2 + \rho u  \left( h'(\rho) + \frac{1}{2}k'(\rho)|\nabla \rho|^2 -\dive (k(\rho)\nabla \rho)\right) \\
&\ +  k(\rho)\nabla \rho\dive (\rho u).
\end{align*}
Before the rigorous justification of the relative entropy calculation in the context of weak solutions we are interested in, let us first briefly present the   (formal) computation leading to the desired expression in the case when both  solutions (of the relaxation and the limiting equations) are regular. 
Let us emphasize  once again that in the sequel  we shall take advantage of the reformulation \eqref{ekb-scaled} in terms of the drift velocity, and the rewriting of the equilibrium equation in \eqref{ekb-scaledstrong}.

If we introduce $m = \rho u$, then (smooth)  solutions of  \eqref{ek} in the diffusive regime satisfy
\begin{align*}
& \partial_t \eta(\rho,m,\nabla \rho) + \frac{1}{\epsilon} \dive \Bigg ( \frac{1}{2}m \frac{|m|^2}{\rho^2} + m \left( h'(\rho) + \frac{1}{2}k'(\rho)|\nabla \rho|^2 
-\dive (k(\rho)\nabla \rho) \right ) +  \\
&\   k(\rho)\nabla \rho\dive m \Bigg ) =
-\frac{1}{\epsilon^2} \frac{|m|^2}{\rho} \leq 0,
\end{align*}
while (smooth) solutions of \eqref{ek-scaledstrong} satisfy the following energy dissipation identity:
\begin{align}\label{etaq-strong}
& \partial_t  \eta(\bar{\rho}, \bar{m}, \nabla \bar \rho    )  + \frac{1}{\epsilon} \dive Q(\bar{\rho}, \bar{m}, \nabla \bar \rho) = - \frac{1}{\epsilon^2} \frac{ |\bar{m}|^2}{\bar
\rho} + \frac{\bar{m}}{\bar{\rho}} \cdot \bar{e}.   
\end{align}
%We refer to \cite{GLT,LT,LT2} for the original derivation. 
It is worth to observe here that  \eqref{etaq-strong} is a rewriting of the classical energy relation valid for the solution $\bar\rho$ to the equilibrium gradient flow equation \eqref{gf}. 
At this point, the main difference here with respect to the arguments in  \cite{GLT,LT,LT2}
  relies on the fact that we use the notation of \cite{Bresh}: 
  we introduce a fictitious velocity $v = \sqrt{\frac{k(\rho)}{\rho}} \nabla \rho$ and correspondingly its transport equation along the velocity $u$ (see \eqref{ekb-scaled}$_3$). 
  This leads us to define a ``new'' entropy-entropy flux pair $(\eta,Q)$ related to the  ``new'' potential energy
  $$F(\rho,J) = h(\rho) + \frac{1}{2} \frac{|J|^2}{\rho},$$
  where $J=\rho v$. 
  Hence, the entropy rewrites as follows:
$$\eta(\rho, m , J) = \frac{1}{2} \frac{|m|^2}{\rho} + h(\rho) + \frac{1}{2} \frac{|J|^2}{\rho},$$
 while its   flux $Q$ is given by:
$$Q(\rho, m, J) = \frac{1}{2} m   \frac{|m|^2}{\rho^2} + mh'(\rho) + \frac{1}{2} m   \frac{|J|^2}{\rho^2}.$$
We get:
\begin{align}\label{single eta}
\partial_t \eta(\rho,m,J) + \frac{1}{\epsilon} \dive Q(\rho,m,J) = \frac{1}{\epsilon} \frac{m}{\rho} \cdot \dive S_1 - \frac{1}{\epsilon} \frac{J}{\rho} \cdot \dive S_2 - \frac{1}{\epsilon^2} \frac{|m|^2}{\rho},
\end{align}
while for the regular solution of the parabolic equation we get:
\begin{align}\label{single etabar}
\partial_t {\eta}(\bar \rho,\bar m, \bar J) + \frac{1}{\epsilon} \dive {Q}(\bar \rho,\bar m, \bar J)&  = \frac{1}{\epsilon} \frac{\bar{m}}{\bar{\rho}} \cdot \dive \bar{S_1} - \frac{1}{\epsilon} \frac{\bar{J}}{\bar{\rho}} \cdot \dive \bar{S_2} - \frac{1}{\epsilon^2} \frac{|\bar{m}|^2}{\bar{\rho}}+ \bar{e} \cdot \frac{\bar{m}}{\bar{\rho}}.
\end{align}
Before formally prove the relative entropy relation in the context of weak solutions, 
here we   sketch the derivation of \eqref{single eta}  for the system $\eqref{ekb-scaled}$ and state the final result. 
To this end, a direct computation shows
\begin{align*}
\partial_t \left(\frac{1}{2} \frac{|m|^2}{\rho} \right) + \frac{1}{\epsilon} \dive \left( \frac{1}{2} m   \frac{|m|^2}{\rho^2} \right) = - \frac{1}{\epsilon} u \cdot \nabla p(\rho) + \frac{1}{\epsilon} u \cdot \dive S_1 - \frac{1}{\epsilon^2} \rho |u|^2,
\end{align*}
and
\begin{align*}
\partial_t F(\rho, J) & = \partial_t \left( h(\rho) + \frac{1}{2} \frac{|J|^2}{\rho} \right)= - \frac{1}{\epsilon} \dive \left( m   \left(h'(\rho) + \frac{1}{2}|v|^2 \right) \right)
\\
&\  + \frac{1}{\epsilon} u \cdot \nabla p(\rho) 
 - \frac{1}{\epsilon} v \cdot \dive S_2,
\end{align*}
leading to \eqref{single eta}.
%\begin{align*}
%\partial_t \left( \eta(\rho, u , v) \right) + \frac{1}{\epsilon} \dive \left( \rho u \cdot \left( \frac{1}{2} u^2 + \frac{1}{2} v^2 + h'(\rho) \right) \right)= \frac{1}{\epsilon} [u \dive S_1 - v \dive S_2] - \frac{1}{\epsilon^2} \rho u^2
%\end{align*}
In this framework,  the relative entropy is defined as:
\begin{align*}
\eta(\rho,m,J|\bar{\rho},\bar{m},\bar{J}) &=  \eta(\rho,m,J) - \eta(\bar{\rho},\bar{m},\bar{J}) - \eta_{\rho}(\bar{\rho},\bar{m},\bar{J})(\rho- \bar{\rho})  
\\
&\ - \eta_{m}(\bar{\rho},\bar{m},\bar{J})\cdot(m-\bar{m}) - \eta_{J}(\bar{\rho},\bar{m},\bar{J})\cdot(J-\bar{J}).
\end{align*}
 When both  solutions  are  regular, it verifies the following relation:
\begin{align*}
& \partial_t \eta (\rho,m,J| \bar{\rho}, \bar{m}, \bar{J}) + \frac{1}{\epsilon} \dive_x Q(\rho,m,J|\bar{\rho}, \bar{m}, \bar{J}   ) = \\ & - \frac{1}{\epsilon} \rho \nabla \bar{u} : (u-\bar{u}) \otimes (u-\bar{u}) - \frac{1}{\epsilon^2} \rho |u-\bar{u}|^2 - \frac{\rho}{\bar{\rho}} \bar{e}\cdot(u-\bar{u}) - \frac{1}{\epsilon}p(\rho| \bar{\rho}) \dive \bar{u} \\ & - \frac{1}{\epsilon} \rho \; \nabla \bar{u} : (v-\bar{v}) \otimes (v-\bar{v})  - \frac{1}{\epsilon}\rho( \mu''(\rho) \nabla \rho - \mu''(\bar{\rho}) \nabla \bar{\rho})) \cdot ((v - \bar{v}) \dive \bar{u} \\
& \ - (u - \bar{u}) \dive \bar{v}) \\ & - \frac{1}{\epsilon} \rho(\mu'(\rho) - \mu'(\bar{\rho}))((v-\bar{v})) \cdot \nabla (\dive\bar{u}) - (u - \bar{u}) \cdot \nabla (\dive \bar{v})),
\end{align*}
where the relative flux is given by
\begin{align*}
Q(\rho,u,v | \bar{\rho}, \bar{u}, \bar{v}) = & \rho u \frac{1}{2}|u-\bar{u}|^2 + \rho u (h'(\rho) - h'(\bar{\rho}) )+ \frac{1}{2} \rho u |v-\bar{v}|^2  \\ &  - \mu(\rho)\nabla v (u-\bar{u}) -  \frac{1}{2} \lambda(\rho)\dive v (u-\bar{u}) - \\ &  \mu(\rho) \nabla u (\bar{v}-v) - \frac{1}{2} \lambda(\rho)\dive u (\bar{v}-v) \\ &  - \mu(\bar{\rho}) \frac{\rho}{\bar{\rho}} \nabla \bar{v} (\bar{u}-u) + \mu(\bar{\rho}) \frac{\rho}{\bar{\rho}}\nabla \bar{u} (\bar{v}-v) \\ & - \rho\left( \frac{\mu(\rho)}{\rho} - \frac{\mu(\bar{\rho})}{\bar{\rho}} \right) (  \nabla \bar{u} (v-\bar{v}) - \nabla \bar{v}(u-\bar{u})) - \\ & \frac{1}{2}\left(\lambda(\rho) - \frac{\rho}{\bar{\rho}}\lambda(\bar{\rho}) \right)((v- \bar{v})\dive \bar{u} - (u-\bar{u})\dive \bar{v}))
\end{align*}
and  the relative entropy can be also rewritten as
$$  \eta( \rho, m,J| \bar{\rho}, \bar{m}, \bar{J}) =   \frac{1}{2} \rho |u - \bar{u}|^2 + \frac{1}{2} \rho |v-\bar{v}|^2 + h(\rho|\bar{\rho}).$$
%
% \subsection{Weak solutions}\label{sec:weaksol}
Now, to generalize this relation for weak solutions, let us first state the precise definition of the latter, based on the one introduced in \cite{LT2}. 
We recall that we shall consider here $\gamma$--law pressures $p(\rho) = \rho^\gamma$, while the capillarity coefficient $k(\rho)$ is given by $k(\rho) = \frac{(s+3)^2}{4} \rho^s$, for which we obtain $\mu(\rho) = \rho^{\frac{s+3}{2}}$, with the conditions $\gamma > 1$, $s+2 \leq \gamma$ and $s \geq -1$.

\begin{definition}\label{ws}
 ($\rho$, $m$, $J$) with $\rho \in C([0, \infty);(L^1(\mathbb{T}^n))$ $(m,J) \in C([0, \infty);(L^1(\mathbb{T}^n))^{2n})$, $\rho \geq 0$, is a weak (periodic) solution of $\eqref{ekb-scaled}$ if 
$$ \sqrt{\rho}u,  \sqrt{\rho}v \in L^{\infty}((0,T);L^2(\mathbb{T}^n)^n),\  \rho \in C([0, \infty);(L^\gamma(\mathbb{T}^n)),$$
 and $(\rho, m,J)$ satisfy for all $\psi \in C^1_c([0, \infty); C^1(\mathbb{T}^n))$ and for all $\phi, \varphi \in C^1_c([0, \infty); C^1(\mathbb{T}^n)^n)$: 
 \begin{align*}
& - \iint_{(0,+\infty)\times\mathbb{T}^n}\Bigg (\rho \psi_t + \frac{1}{\epsilon} m \cdot \nabla_x \psi \Bigg )dxdt = \int_{\mathbb{T}^n} \rho(x,0)\psi(x,0);\\
&  -  \iint_{(0,+\infty)\times\mathbb{T}^n} \Bigg[m \cdot (\phi)_t + \frac{1}{\epsilon}\left(\frac{m \otimes m}{\rho} : \nabla_x \phi \right) + \frac{1}{\epsilon} p(\rho) \dive \phi  \\
& + \frac{1}{\epsilon}\left( \mu(\rho) v \cdot \nabla \dive (\phi) + \nabla \mu(\rho) \cdot (\nabla \phi v ) + \frac{1}{2} \nabla \lambda(\rho) \cdot v \dive \phi + \frac{1}{2} \lambda(\rho) v \cdot \nabla \dive \phi \right)  \Bigg]dxdt  \\ 
  & \ = - \frac{1}{\epsilon^2} \iint_{(0,+\infty)\times\mathbb{T}^n}m \cdot \phi dxdt + \int_{\mathbb{T}^n} m(x,0) \cdot\phi(x,0)dx,
\end{align*}
where we have used the identity 
$$\displaystyle{S= - p(\rho) \mathbb{I} + S_1= - p(\rho)\mathbb{I} + \mu(\rho)\nabla v + \frac{1}{2}\lambda(\rho)\dive v \mathbb{I}};$$
\begin{align*}
& - \iint_{(0,+\infty)\times\mathbb{T}^n} \Bigg[J \cdot\varphi_t + \frac{1}{\epsilon}\left(\frac{J \otimes m}{\rho} : \nabla_x \varphi \right) - \frac{1}{\epsilon}\Bigg( \mu(\rho) u \cdot   (\nabla \dive \varphi ) + \nabla \mu(\rho) \cdot (\nabla \varphi u ) \\
&\ + \frac{1}{2} \nabla \lambda(\rho) \cdot u \dive \varphi + \frac{1}{2}  \lambda(\rho) u \cdot  \nabla \dive \varphi \Bigg)  \Bigg]dxdt =  \int_{\mathbb{T}^n} J(x,0) \cdot \varphi(x,0)dx,
\end{align*}
where we have used the identity 
$$\displaystyle{S_2= \mu(\rho)^t\nabla u + \frac{1}{2}\lambda(\rho)\dive u \mathbb{I}}.$$

If in addition $ \eta (\rho,m,J) \in C([0, \infty); L^1(\mathbb{T}^n))$ and $(\rho,m,J)$ satisfy
\begin{align}\label{diss}
& \iint_{(0,+\infty)\times\mathbb{T}^n} \left( \eta(\rho,m,J) \right) \dot{\theta}(t) dxdt \leq  \int_{\mathbb{T}^n} \left( \eta(\rho,m,J)\right)|_{t=0} \theta(0)dx
\nonumber\\ 
&\ - \frac{1}{\epsilon^2}  \iint_{(0,+\infty)\times\mathbb{T}^n} \frac{|m|^2}{\rho} \theta(t) dxdt
\end{align}
for any non-negative $\theta \in W^{1,\infty}[0, \infty)$ compactly supported on $[0,\infty)$, then $(\rho,m,J)$ is called a \emph{dissipative} weak solution.

If $\eta(\rho,m,J) \in C([0,\infty);L^1(\mathbb{T}^n))$ and $(\rho,m,J)$ satisfy $\eqref{diss}$ as an equality, then $(\rho,m,J)$ is called a \emph{conservative} weak solution.

We say that a dissipative (or conservative) weak (periodic) solution $(\rho,m,J)$ of $\eqref{ekb-scaled}$ with $\rho \geq 0$ has finite total mass and energy if
$$ \sup_{t \in (0,T)} \int_{\mathbb{T}^n} \rho dx \leq M <+ \infty,$$
and
$$ \sup_{t \in (0,T)} \int_{\mathbb{T}^n}  \eta(\rho,m,J) dx \leq E_o <+ \infty.$$
\end{definition}
\begin{theorem}\label{relativentropy}
Let $(\rho,m,J)$ be a dissipative (or conservative) weak solution of $\eqref{ekb-scaled}$ with finite total mass and energy according to Definition \ref{ws}, and let $\bar{\rho}$ be a smooth solution of $\eqref{gf}$. Then
\begin{align}\label{eq:relen}
& \int_{\mathbb{T}^n} \eta(\rho,m,J| \bar{\rho}, \bar{m}, \bar{J})(t)dx \leq \int_{\mathbb{T}^n} \eta(\rho,m,J| \bar{\rho}, \bar{m}, \bar{J})(0)dx
\nonumber\\ 
& - \frac{1}{\epsilon^2} \iint_{(0,t) \times \mathbb{T}^n}  \rho |u-\bar{u}|^2 dxd\tau  - \frac{1}{\epsilon} \iint_{(0,t) \times \mathbb{T}^n} \rho \nabla \bar{u}: (u-\bar{u}) \otimes (u-\bar{u})dxdt 
\nonumber\\ 
& - \frac{1}{\epsilon} \iint_{(0,t) \times \mathbb{T}^n}p(\rho|\bar{\rho}) \dive \bar{u} dxd\tau  
- \frac{1}{\epsilon} \iint_{(0,t) \times \mathbb{T}^n} \rho \; \nabla \bar{u}: (v-\bar{v}) \otimes (v-\bar{v}) dxd\tau 
\nonumber\\ 
& - \iint_{(0,t) \times \mathbb{T}^n} e(\bar{\rho},\bar{m}) \cdot \frac{\rho}{\bar{\rho}} (u-\bar{u})dxd\tau 
\nonumber\\ 
& - \frac{1}{\epsilon} \iint_{(0,t) \times \mathbb{T}^n} \rho[(\mu''(\rho)\nabla \rho - \mu''(\bar{\rho})\nabla \bar{\rho})\cdot((v-\bar{v})\dive \bar{u} - (u-\bar{u})\dive \bar{v})]dxd\tau \nonumber\\ 
& - \frac{1}{\epsilon}\iint_{(0,t) \times \mathbb{T}^n} \rho (\mu'(\rho)- \mu'(\bar{\rho}))[(v-\bar{v})\cdot\nabla \dive \bar{u}-(u- \bar{u})\cdot\nabla \dive \bar{v}]dxd\tau,
\end{align}
where 
\begin{equation}\label{eq:defbarutheo}
\bar{m} = \bar{\rho}\bar{u}= \epsilon \left(- \nabla p(\bar{\rho})+   \dive S_1(\bar{\rho})\right); \ \bar{J}= \bar{\rho}\bar{v} = \nabla \mu(\bar{\rho}).
\end{equation}
\end{theorem}
\begin{proof}
Let $(\rho,m,J)$ be a weak dissipative (or conservative) weak solution of $\eqref{ekb-scaled}$ according to  Definition \ref{ws} and let $\bar \rho$ 
 be a strong solution of $\eqref{gf}$, so that, using \eqref{eq:defbarutheo}, $(\bar{\rho}, \bar{m},\bar{J})$ satisfies $\eqref{ekb-scaledstrong}$. We consider the following function $\theta(\tau)$ in the energy (in)equality  \eqref{diss} of  Definition \ref{ws}:
\begin{equation*}
\theta(\tau)=	
\begin{cases}
		1,    & \hbox{ for } 0 \leq \tau < t, \\
		  \frac{t-\tau}{\mu} + 1, & \hbox{ for } t \leq \tau < t+\tau, \\
	 0, & \hbox{ for } \tau \geq t +\mu.
		\end{cases}
\end{equation*}
Then, as $\mu \rightarrow 0$,  we readily obtain:
\begin{align*}
\int_{\mathbb{T}^n} (\eta(\rho,m,J))|_{\tau =0}^t \leq - \frac{1}{\epsilon^2} \iint_{(0,t)\times\mathbb{T}^n} \frac{|m|^2}{\rho} dxd\tau.
\end{align*}
Moreover, by a direct integration in $(0,t) \times \mathbb{T}^n$ of  $\eqref{single etabar}$   we get:
\begin{equation}
\begin{split}
\int_{\mathbb{T}^n} \eta(\bar{\rho},\bar{m}, \bar{J})|_{\tau = 0}^t = & - \frac{1}{\epsilon^2} \iint_{(0,t)\times\mathbb{T}^n} \frac{|\bar{m}|^2}{\bar{\rho}}dxd\tau
+ \iint_{(0,t)\times\mathbb{T}^n} \frac{\bar{m}}{\bar{\rho}} \cdot \bar{e} 
\end{split}
\end{equation}
because 
\begin{align}\label{rhobar}
0 &= - \frac{1}{\epsilon}\iint_{(0,t)\times\mathbb{T}^n}\big (\nabla \bar{u} : \bar{S_1} - \nabla \bar{v} :  \bar{S_2} \big ) dxd\tau \nonumber\\
& = \frac{1}{\epsilon}\iint_{(0,t)\times\mathbb{T}^n}\big (\bar{u} \cdot \dive \bar{S_1} - \bar{v} \cdot \dive \bar{S_2} \big ) dxd\tau \nonumber\\ 
&= \iint_{(0,t)\times\mathbb{T}^n} \big (\nabla \mu (\bar \rho) \cdot (\nabla \bar v \bar u - \nabla \bar u \bar v) + 
\mu(\bar \rho) (\bar u \cdot \nabla \dive \bar v - \bar v \cdot \nabla \dive \bar u)\big)dxd\tau \nonumber\\
& \  + \frac{1}{2\epsilon} \iint_{(0,t)\times\mathbb{T}^n} \big( \nabla\lambda(\bar \rho)\cdot(\bar u  \dive \bar v - \bar v \dive \bar u) + \lambda(\bar \rho)(\bar u \cdot \nabla \dive \bar v  - \bar v \cdot  \nabla \dive \bar u)\big) dxd\tau,
\end{align}  
being $\bar\rho $ periodic and using  the definitions of $\bar{S_1}$ and $\bar{S_2}$:
\begin{align*}
&\bar{S_1}=  \mu(\bar\rho)\nabla \bar v + \frac{1}{2}\lambda(\bar\rho)\dive \bar v \mathbb{I}; \\
& \bar{S_2} = \mu(\bar\rho)^t\nabla \bar u + \frac{1}{2}\lambda(\bar\rho)\dive \bar u \mathbb{I}.
\end{align*}
Indeed we have:
%\begin{align*}
%&\iint_{(0,t)\times\mathbb{T}^n}\nabla \bar{u} :  \bar{S_1}  dxd\tau = \iint_{(0,t)\times\mathbb{T}^n}\mu(\bar{\rho}) \nabla \bar{u} : \nabla \bar{v} dxd\tau \\ 
%%& +  \iint_{(0,t)\times\mathbb{T}^n} \frac{1}{2} \lambda(\bar{\rho}) \dive \bar{v} \nabla \bar{u} : \mathbb{I} dxd\tau \\
%& = \iint_{(0,t)\times\mathbb{T}^n} \mu(\bar{\rho}) \nabla \bar{u} : \nabla \bar{v} +  \frac{1}{2} \lambda(\bar{\rho}) \dive \bar{v} \dive \bar{u} dxd\tau,
%\end{align*}
\begin{equation*}
\frac{1}{\epsilon}\iint_{(0,t)\times\mathbb{T}^n}\nabla \bar{u} :  \bar{S_1}  dxd\tau = % \iint_{(0,t)\times\mathbb{T}^n}\mu(\bar{\rho}) \nabla \bar{u} : \nabla \bar{v} dxd\tau \\ 
%& +  \iint_{(0,t)\times\mathbb{T}^n} \frac{1}{2} \lambda(\bar{\rho}) \dive \bar{v} \nabla \bar{u} : \mathbb{I} dxd\tau \\
 \frac{1}{\epsilon} \iint_{(0,t)\times\mathbb{T}^n}\left ( \mu(\bar{\rho}) \nabla \bar{u} : \nabla \bar{v} +  \frac{1}{2} \lambda(\bar{\rho}) \dive \bar{v} \dive \bar{u} \right )dxd\tau,
\end{equation*}
and
\begin{equation*}
\frac{1}{\epsilon}\iint_{(0,t)\times\mathbb{T}^n}\nabla \bar{v} :  \bar{S_2}  dxd\tau %= \iint_{(0,t)\times\mathbb{T}^n}\mu(\bar{\rho}) \nabla \bar{v} : {}^t\nabla \bar{u} dxd\tau \\ 
%& +  \iint_{(0,t)\times\mathbb{T}^n} \frac{1}{2} \lambda(\bar{\rho}) \dive \bar{u} \nabla \bar{v} : \mathbb{I} dxd\tau \\
 = \frac{1}{\epsilon}\iint_{(0,t)\times\mathbb{T}^n} \left (\mu(\bar{\rho}) \nabla \bar{v} : {}^t\nabla \bar{u} +  \frac{1}{2} \lambda(\bar{\rho}) \dive \bar{u} \dive \bar{v} \right )dxd\tau.
\end{equation*}
%\begin{align*}
%&\iint_{(0,t)\times\mathbb{T}^n}\nabla \bar{v} :  \bar{S_2}  dxd\tau = \iint_{(0,t)\times\mathbb{T}^n}\mu(\bar{\rho}) \nabla \bar{v} : {}^t\nabla \bar{u} dxd\tau \\ 
%%& +  \iint_{(0,t)\times\mathbb{T}^n} \frac{1}{2} \lambda(\bar{\rho}) \dive \bar{u} \nabla \bar{v} : \mathbb{I} dxd\tau \\
%& = \iint_{(0,t)\times\mathbb{T}^n} \mu(\bar{\rho}) \nabla \bar{v} : {}^t\nabla \bar{u} +  \frac{1}{2} \lambda(\bar{\rho}) \dive \bar{u} \dive \bar{v} dxd\tau.
%\end{align*}
Therefore, since 
$$\nabla \bar{v} = \nabla\left ( \frac{\mu'(\bar\rho)}{\bar\rho}\nabla\bar\rho\right ) = \nabla^2 M(\bar\rho)
$$
 is symmetric,
 it holds: 
 $$\nabla \bar{u}: \nabla \bar{v}- \nabla \bar{v}: {}^t\nabla \bar{u} = \nabla \bar{u}: \nabla \bar{v}- \nabla \bar{u}: {}^t\nabla \bar{v} = \nabla \bar{u}: \nabla \bar{v}- \nabla \bar{u}: \nabla \bar{v} = 0,$$
and the integral
\begin{align*}
& \frac{1}{\epsilon}\iint_{(0,t)\times\mathbb{T}^n}\big ( \nabla \bar{u} : \bar{S_1} - \nabla \bar{v} :   \bar{S_2} \big )dxd\tau  \\
& \ = \frac{1}{\epsilon}\iint_{(0,t)\times\mathbb{T}^n}\left ( \mu(\bar{\rho}) (\nabla \bar{u}: \nabla \bar{v}- \nabla \bar{v}: {}^t\nabla \bar{u}) + \frac{1}{2} \lambda(\bar{\rho}) (\dive \bar{v}\dive \bar{u} - \dive\bar{v}\dive\bar{u}) \right )dxd\tau 
\end{align*}
vanishes. 
%On the other hand, by a direct computation we have:
%\begin{equation}\label{udivS1-vdivS2}
%\begin{split}
%& \frac{1}{\epsilon} \iint_{(0,t)\times\mathbb{T}^n} \bar u \dive \bar{S_1} - \bar v \dive \bar{S_2} dxd\tau =  \\ & \textcolor{green}{+ \frac{1}{\epsilon} \iint_{(0,t)\times\mathbb{T}^n} \nabla \mu (\bar \rho) \cdot (\nabla \bar v \bar u - \nabla \bar u \bar v) + \mu(\bar \rho) (\bar u \cdot \nabla \dive \bar v - \bar v \cdot \nabla \dive \bar u)} + \\ & \textcolor{green}{+ \frac{1}{2\epsilon} \iint_{(0,t)\times\mathbb{T}^n} \bar u \cdot \nabla(\lambda(\bar \rho)) \dive \bar v + \lambda(\bar \rho) \nabla \dive \bar v \cdot \bar u - \bar v \cdot \nabla(\lambda(\bar \rho)) \dive \bar u - \lambda(\bar \rho) \bar v \cdot  \nabla \dive \bar u dxd\tau}
%\end{split}
% \end{equation}

Now we want to evaluate the linear part of the relative entropy for the difference $(\rho - \bar{\rho}, m - \bar{m}, J- \bar{J})$ choosing suitable test functions in the weak formulation (according to  Definition \ref{ws}) of the equation satisfied by these differences, namely:
 \begin{equation}\label{psi}
-  \iint_{[0,\infty) \times \mathbb{T}^n} \left( \psi_t(\rho- \bar{\rho}) + \frac{1}{\epsilon} \psi_{x_i}(m_i-\bar{m_i}) \right)dxd\tau = \int_{\mathbb{T}^n} (\rho - \bar{\rho})\psi|_{t=0}dx,
 \end{equation}
  % - \int_{T^n} \phi \cdot(m_1-\bar{m_1})|_{t=0}dx
\begin{equation}\label{phi}
\begin{split}
& - \iint_{[0,\infty) \times \mathbb{T}^n} \phi_{t} \cdot (m-\bar{m}) + \frac{1}{\epsilon}  \left( \frac{m_i m_j}{\rho} - \frac{\bar{m_i}\bar{m_j}}{\bar{\rho}} \right)\partial_{x_j}\phi_{i} + \frac{1}{\epsilon} [p(\rho)- p(\bar{\rho})] \partial_{x_i} \phi_{i} dxd\tau \\ &  
- \frac{1}{\epsilon} \iint_{[0,\infty) \times \mathbb{T}^n} (\mu(\rho)v_i-\mu(\bar\rho)\bar{v_i}) \partial_{x_i}  \partial_{x_j} \phi_{j} + (\partial_{x_i} \mu(\rho)v_j - \partial_{x_i}\mu(\bar\rho) 
\bar{v_j})\partial_{x_j} \phi_{i}
dxd\tau 
 \\ & - \frac{1}{\epsilon} \iint_{[0,\infty) \times \mathbb{T}^n} \frac{1}{2}( \partial_{x_i}( \lambda(\rho)) v_i   - \partial_{x_i}(\lambda(\bar \rho)) \bar{v_i}) \partial_{x_j}\phi_{j} + \frac{1}{2} (\lambda(\rho) v_i-
 \lambda(\bar\rho)\bar{v_i}) \partial_{x_i} \partial_{x_j}\phi_{j} dxd\tau  \\ 
 & \ =  - \frac{1}{\epsilon^2} \iint_{[0, \infty) \times \mathbb{T}^n} (m - \bar{m}) \cdot \phi dx d\tau - \iint_{[0, \infty) \times T^n} \bar{e} \cdot \phi dx d\tau + \int_{\mathbb{T}^n}  
  (m - \bar m)\cdot \phi|_{t=0}dx
\end{split}
\end{equation}
and
\begin{equation}\label{phi2}
\begin{split}
& - \iint_{[0,\infty) \times \mathbb{T}^n} \left( \varphi_t \cdot (J-\bar{J}) \right) + \frac{1}{\epsilon} \left( \frac{J_{i} m_{j}}{\rho} - \frac{\bar{J_{i}}\bar{m_{j}}}{\bar{\rho}} \right) \partial_{x_j}\varphi_i dxd\tau \\ 
& + \frac{1}{\epsilon} \iint_{[0,\infty) \times \mathbb{T}^n} (\mu(\rho)u_i-\mu(\bar\rho)\bar{u_i}) \partial_{x_i}  \partial_{x_j} \varphi_j   + 
(\partial_{x_i} \mu(\rho)u_j-\partial_{x_i}\mu(\bar\rho)\bar{u_j})\partial_{x_j} \varphi_i dxd\tau \\ 
& + \frac{1}{\epsilon} \iint_{[0,\infty) \times \mathbb{T}^n} \frac{1}{2}(\partial_{x_i} \lambda(\rho)u_i-\partial_{x_i} \lambda(\bar\rho)\bar{u_i})  \partial_{x_j} \varphi_j + \frac{1}{2}(\lambda(\rho)u_i- \lambda(\bar\rho)\bar{u_i})  \partial_{x_i} \partial_{x_j} \varphi_j   dxd\tau \\
&\ = \int_{\mathbb{T}^n}   (J - \bar J)\cdot \varphi|_{t=0}dx,
\end{split}
\end{equation}
where $\psi,\phi, \varphi$ are Lipschitz test functions, $\phi, \varphi$ vector--valued, compactly supported in $[0, + \infty)$ in time and periodic in space. In the above relation we choose in particular  
\begin{align*}
&\psi = \theta(\tau) \left( h'(\bar{\rho})- \frac{1}{2} \frac{|\bar{m}|^2}{\bar{\rho^2}} - \frac{|\bar{J}|^2}{ \bar{\rho^2}} \right) \text{ \; and} \\
& \Phi = (\phi, \varphi)= \theta(\tau) \left( \frac{\bar{m}}{\bar{\rho}}, \frac{\bar{J}}{\bar{\rho}} \right), \text{where $\theta(\tau)$ is defined above.}
\end{align*}
Then, letting $\mu \rightarrow 0$ in $\eqref{psi}$  we obtain
\begin{align*}
& \int_{\mathbb{T}^n}\left( h'(\bar{\rho})- \frac{1}{2} \frac{|\bar{m}|^2}{\bar{\rho^2}} - \frac{|\bar{J}|^2}{ \bar{\rho^2}} \right)(\rho - \bar{\rho} ) \mid_{\tau = 0}^t dx \\ 
& - \iint_{[0,t] \times \mathbb{T}^n} \left[ \partial_{\tau} \left(h'(\bar{\rho}) - \frac{1}{2} \frac{|\bar{m}|^2}{\bar{\rho^2}} - \frac{|\bar{J}|^2}{ \bar{\rho^2}} \right) (\rho- \bar{\rho}) dx d\tau \right] 
\\ & - \frac{1}{\epsilon} \iint_{[0,t] \times \mathbb{T}^n} \nabla_x \left(h'(\bar{\rho}) - \frac{|\bar{m}|^2}{ \bar{\rho^2}} - \frac{|\bar{J}|^2}{ \bar{\rho^2}} \right) \cdot (m- \bar{m})dxd\tau = 0.
\end{align*}
From $\eqref{phi}$:
\begin{align*}
%\begin{split}
& \int_{\mathbb{T}^n} \frac{\bar{m}}{\bar{\rho}} \cdot (m - \bar{m})|_{\tau=0}^tdx - \iint_{[0,t] \times \mathbb{T}^n} \partial_{\tau} \left( \frac{\bar{m}}{\bar{\rho}} \right) \cdot (m-\bar{m}) dxd\tau
\\ &  - \frac{1}{\epsilon} \iint_{[0,t] \times \mathbb{T}^n} \left [  \left( \frac{m_{i} m_{j}}{ \rho}   - \frac{ \bar{m_{i}} \bar{m_{j}} }{ \bar{\rho} } \right) \partial_{x_j}\left( \frac{\bar{m_{i}}}{\bar{\rho}} \right) + (p(\rho)- p(\bar{\rho})) \dive \left( \frac{\bar{m}}{\bar{\rho}} \right) \right] dxd\tau 
\\ 
& - \frac{1}{\epsilon} \iint_{[0,t] \times \mathbb{T}^n}\left [ \mu(\rho)(v-\bar{v}) \cdot \nabla\dive \left( \frac{\bar{m}}{\bar{\rho}} \right)
+ \nabla \mu(\rho)\cdot \nabla \left( \frac{\bar{m}}{\bar{\rho}} \right)(v-\bar{v}) \right ]dxd\tau 
\\ 
& - \frac{1}{\epsilon} \int \int_{[0,t] \times \mathbb{T}^n} \left [(\mu(\rho) - \mu(\bar\rho)) \bar v \cdot \nabla\dive \left( \frac{\bar{m}}{\bar{\rho}} \right) +
\nabla( \mu(\rho) -\mu(\bar\rho))  \nabla \left( \frac{\bar{m}}{\bar{\rho}} \right) \bar v \right ] dxd\tau
\\ 
&  - \frac{1}{2\epsilon} \iint_{[0,\infty) \times \mathbb{T}^n} \left( \nabla \lambda(\rho)\cdot (v-\bar{v}) \dive \left( \frac{\bar{m}}{\bar{\rho}} \right) +  \lambda(\rho)(v-\bar{v}) \cdot \nabla 
\dive \left( \frac{\bar{m}}{\bar{\rho}} \right) \right )dxd\tau 
\\ 
& 
-\frac{1}{ 2 \epsilon} \iint_{[0,\infty) \times \mathbb{T}^n} \left [ \nabla(\lambda(\rho) - \lambda(\bar \rho)) \cdot \bar v \dive \left( \frac{\bar{m}}{\bar{\rho}} \right) +  (\lambda(\rho) - 
\lambda(\bar \rho)) \bar v \cdot \nabla  \dive \left( \frac{\bar{m}}{\bar{\rho}} \right) \right ]dxd\tau 
 \\ 
&=- \frac{1}{\epsilon^2} \iint_{[0,t] \times \mathbb{T}^n} \frac{\bar{m} }{ \bar{\rho}} \cdot (m -\bar{m}) dxd\tau - \iint_{[0,t] \times \mathbb{T}^n} \frac{\bar{m}}{\bar{\rho}} \cdot \bar{e} dxd\tau.
%\end{split}
\end{align*}
Analogously,  from $\eqref{phi2}$:
\begin{align*}
%\begin{split}
& \int_{\mathbb{T}^n} \frac{\bar{J}}{\bar{\rho}} \cdot (J - \bar{J})|_{\tau=0}^tdx - \iint_{[0,t] \times \mathbb{T}^n} \partial_{\tau} \left( \frac{ \bar{J} }{ \bar{\rho}} \right) \cdot (J - \bar{J}) dxd\tau 
\\ & 
- \frac{1}{\epsilon} \iint_{[0,t] \times \mathbb{T}^n}   \left( \frac{m_{j} J_{i} }{\rho} - \frac{ \bar{m_{j}}  \bar{ J_{i}  } }{\bar{\rho}} \right)\partial_{x_j} \left( \frac{\bar{J_{i}} }{\bar{\rho}} \right) dxd\tau  
 \\ & 
 + \frac{1}{\epsilon} \iint_{[0,t] \times \mathbb{T}^n}\left [ \mu(\rho)\left (\frac{m}{\rho} - \frac{\bar{m}}{\bar{\rho}}\right ) \cdot \nabla (\dive \bar{v}) + \nabla \mu(\rho) \cdot \nabla \bar{v}
\left (\frac{m}{\rho} - \frac{\bar{m}}{\bar{\rho}}\right )  \right ]dxd\tau 
 \\
 & 
 + \frac{1}{\epsilon} \iint_{[0, t) \times \mathbb{T}^n} \left [(\mu(\rho) - \mu(\bar\rho) ) \frac{\bar{m}}{\bar{\rho}}\cdot \nabla \dive \bar v + \nabla( \mu(\rho) -\mu(\bar\rho)) \nabla \bar v \frac{\bar{m}}{\bar{\rho}} \right ] dxd\tau
 \\ & 
 + \frac{1}{2\epsilon} \iint_{[0,t] \times \mathbb{T}^n} \left [ \nabla \lambda(\rho) \cdot \left (\frac{m}{\rho} - \frac{\bar{m}}{\bar{\rho}}\right )\dive \bar{v} + \lambda(\rho)
 \left (\frac{m}{\rho} - \frac{\bar{m}}{\bar{\rho}}\right ) \cdot \nabla \dive \bar{v} \right ]dxd\tau 
 \\ & 
+ \frac{1}{2 \epsilon} \iint_{[0,t] \times \mathbb{T}^n} \nabla(\lambda(\rho) -\lambda(\bar\rho)) \cdot \frac{\bar{m}}{\bar{\rho}} \dive \bar v + (\lambda(\rho) -\lambda(\bar{\rho})) 
\frac{\bar{m}}{\bar{\rho}}  \cdot \nabla \dive \bar v dxd\tau  = 0.
%\end{split}
\end{align*}
%We recall that:
%\begin{equation*}
%\begin{split}
%& {}^t \nabla \bar v = \nabla \bar v ,\\
%& {}^t(u-\bar u) {}^t \nabla \bar v = \nabla \bar v (u-\bar u) 
%\end{split}
%\end{equation*}
Combining the above relations we get:
\begin{align}\label{corradded}
%\begin{split}
& \int_{\mathbb{T}^n} \left[ \eta(\rho,m,J| \bar{\rho}, \bar{m}, \bar{J}) \right]|_{\tau=0}^t dx \leq   - \frac{1}{\epsilon^2} \iint_{[0,t]\times \mathbb{T}^n} [ \rho |u|^2 - \bar{\rho}|\bar{u}|^2 - \bar{u}(\rho u - \bar{\rho} \bar{u})]dxd\tau
\nonumber\\ 
 & - \iint_{[0,t]\times \mathbb{T}^n} \left[ \partial_{\tau} \left(h'(\bar{\rho}) - \frac{1}{2}|\bar{u}|^2 - \frac{1}{2}|\bar{v}|^2 \right)(\rho - \bar{\rho}) + \partial_{\tau}(\bar{u})(\rho u - \bar{\rho} \bar{u}) + \partial_{\tau}(\bar{v})(\rho v - \bar{\rho} \bar{v}) \right]dxd\tau
 \nonumber \\ 
  & - \frac{1}{\epsilon} \iint_{[0,t]\times \mathbb{T}^n} \nabla \left( h'(\bar{\rho}) - \frac{1}{2}|\bar{u}|^2 - \frac{1}{2}|\bar{v}|^2 \right)(\rho u - \bar{\rho} \bar{u}) dxd\tau - \frac{1}{\epsilon} \iint_{[0,\infty) \times \mathbb{T}^n} [p(\rho)-p(\bar{\rho})] \dive \bar{u} dxd\tau 
  \nonumber\\ 
  & - \frac{1}{\epsilon} \iint_{[0,t]\times \mathbb{T}^n} (\rho u_i u_j - \bar{\rho} \bar{u_i} \bar{u_j})\partial_{x_j}(\bar{u_i})  + (\rho v_i u_j - \bar{\rho} \bar{v_i} \bar{u_j})\partial_{x_j}(\bar{v_i})  dxd\tau  
  \nonumber\\ 
  & - \frac{1}{\epsilon} \iint_{[0,t]\times \mathbb{T}^n} \mu(\rho)[(v-\bar{v})  \nabla \dive \bar{u} - (u-\bar{u}) \nabla \dive \bar{v}] + \nabla \mu(\rho)[
\nabla \bar{u}(v-\bar{v}) - \nabla \bar{v}(u-\bar{u})]dxd\tau 
\nonumber\\ 
& - \frac{1}{2\epsilon} \iint_{[0,t]\times \mathbb{T}^n}  \nabla \lambda(\rho)[ (v-\bar{v}) \dive \bar{u} - (u-\bar{u}) \dive \bar{v}] +  \lambda(\rho)[(v-\bar{v}) \nabla \dive \bar{u} - (u-\bar{u}) \nabla \dive \bar{v}]dxd\tau 
\nonumber\\
& +
 \frac{1}{\epsilon} \iint_{[0,t] \times \mathbb{T}^n} \big( (\mu(\rho) - \mu(\bar\rho))(\bar u \cdot\nabla \dive \bar v - \bar v \cdot\nabla \dive \bar u ) + \nabla (\mu(\rho)-\mu(\bar\rho))\cdot(\nabla \bar v \bar u - \nabla \bar u \bar v )\big )dxd\tau 
\nonumber\\ 
& +\frac{1}{2 \epsilon}\iint_{[0,t] \times \mathbb{T}^n} \big((\lambda(\rho) -\lambda(\bar\rho)) \left(\bar u \cdot \nabla \dive \bar v  - \bar v \cdot \nabla \dive \bar u  \right) + \nabla(\lambda(\rho)-\lambda(\bar\rho))\cdot (\bar u \dive \bar v - \bar v  \dive \bar u )\big)dxd\tau.
%\end{split}
\end{align}

First of all, let us 
observe that the last two lines of the relation above are indeed zero, as one can easily prove repeating the arguments leading to \eqref{rhobar}, with the differences $\mu(\rho)-\mu(\bar\rho)$ and $\lambda(\rho)-\lambda(\bar\rho)$ replacing $\mu(\bar\rho)$ and $\lambda(\bar\rho)$ inside the definition of the tensors $S_1$ and $S_2$. Moreover, using the relation $h''(\bar{\rho}) = p'(\bar{\rho})/ \bar{\rho}$ and the continuity equation for $\bar\rho$, we get
\begin{align*}
& - \iint_{[0,t]\times \mathbb{T}^n} \left ( \partial_{\tau}h'(\bar{\rho})(\rho- \bar{\rho}) + \frac{1}{\epsilon} \nabla h'(\bar{\rho}) (\rho u - \bar{\rho}\bar{u}) \right) dxd\tau  \\ 
& =
\frac{1}{\epsilon} \iint_{[0,t]\times \mathbb{T}^n} \left ( p'(\bar{\rho})(\rho- \bar{\rho})\dive \bar{u} +  \frac{\rho}{\bar{\rho}} \nabla p(\bar{\rho})(\bar{u}- u) \right )dxd\tau.
\end{align*}
We multiply the transport equations of $\bar{u}$ and $\bar{v}$, namely
\begin{align*}
& \partial_\tau \bar{u} + \frac{1}{\epsilon} ( \bar{u} \cdot \nabla \bar{u}) + \frac{1}{\epsilon} \frac{\nabla p(\bar{\rho})}{\bar{\rho}} - \frac{\dive \bar{S_1}}{\bar{\rho}} - \frac{\bar{e}}{\bar{\rho}}= - \frac{1}{\epsilon^2} \bar{u}, 
\\ 
& \partial_\tau \bar{v} + \frac{1}{\epsilon} (\bar{u} \cdot \nabla \bar{v}) + \frac{1}{\epsilon} \frac{\dive \bar{S_2}}{\bar{\rho}} = 0,
\end{align*}
 by $\rho( \bar{u} - u)$ and $\rho( \bar{v} - v)$ respectively to conclude
% we end up with 
% \begin{align}\label{equ}
% &\partial_t \left( \frac{1}{2} |\bar{u}|^2 \right)(\rho - \bar{\rho}) + \frac{1}{\epsilon} \nabla \left( \frac{1}{2}|\bar{u}|^2 \right)(\rho u - \bar{\rho} \bar{u}) - \partial_t \bar{u}(\rho u - \bar{\rho} \bar{u}) - \frac{1}{\epsilon} \partial_{x_j} \bar{u_i
% }( \rho u_iu_j - \bar{\rho} \bar{u_i} \bar{u_j} ),
% \end{align}
% and
% \begin{align}\label{eqv}
% \partial_t \left( \frac{1}{2} |\bar{v}|^2 \right)(\rho - \bar{\rho}) + \frac{1}{\epsilon} \nabla \left( \frac{1}{2}|\bar{v}|^2 \right)(\rho u - \bar{\rho} \bar{u}) - \partial_t \bar{v}(\rho u - \bar{\rho} \bar{u}) - \frac{1}{\epsilon} \partial_{x_j} \bar{v_i}( \rho v_iu_j - \bar{\rho} \bar{v_i} \bar{u_j}).
% \end{align}
% The terms in $\eqref{equ}$ can be recognized by multiplying the equation for $\bar{u}$ 
% by $\rho( \bar{u} - u)$, we get:
\begin{align*}
&\partial_\tau \left( \frac{1}{2} |\bar{u}|^2 \right)(\rho - \bar{\rho}) + \frac{1}{\epsilon} \nabla \left( \frac{1}{2}|\bar{u}|^2 \right)\cdot (\rho u - \bar{\rho} \bar{u}) - \partial_\tau \bar{u}\cdot (\rho u - \bar{\rho} \bar{u}) - \frac{1}{\epsilon} \partial_{x_j} \bar{u_i}( \rho u_iu_j - \bar{\rho} \bar{u_i} \bar{u_j})
\\ 
& = - \frac{1}{\epsilon} \rho \nabla \bar{u}:[(u- \bar{u}) \otimes (u-\bar{u})] - \frac{1}{\epsilon} \frac{\nabla p(\bar{\rho})}{\bar{\rho}} \rho \cdot (\bar{u}-u) + \frac{1}{\epsilon} \frac{\rho}{\bar{\rho}} \dive \bar{S_1}\cdot(\bar{u}- u) + \bar{e} \frac{\rho}{\bar{\rho}}\cdot( \bar{u} - u) 
\\ 
& \ - \frac{1}{\epsilon^2} \rho \bar{u}\cdot(\bar{u}-u)
\end{align*}
% We do the same for $\eqref{eqv}$ multiplying $\rho( \bar{v} - v)$  for and we obtain:
and
\begin{align*}
& \partial_\tau \left( \frac{1}{2} |\bar{v}|^2 \right)(\rho - \bar{\rho}) + \frac{1}{\epsilon} \nabla \left( \frac{1}{2}|\bar{v}|^2 \right)\cdot(\rho u - \bar{\rho} \bar{u}) - \partial_\tau \bar{v}\cdot(\rho u - \bar{\rho} \bar{u}) - \frac{1}{\epsilon} \partial_{x_j} \bar{v_i}( \rho v_iu_j - \bar{\rho} \bar{v_i} \bar{u_j})  
\\ 
& = - \frac{1}{\epsilon} \rho \nabla \bar{v} :[(v-\bar v) \otimes ( u - \bar u)] - \frac{1}{\epsilon} \frac{\rho}{\bar{\rho}} \dive \bar{S_2} \cdot (\bar{v}-v).
\end{align*}
In view of the calculation above, \eqref{corradded} rewrites as follows:
\begin{align}\label{corradded2}
& \int_{\mathbb{T}^n} \left[ \eta(\rho,m,J| \bar{\rho}, \bar{m}, \bar{J}) \right]|_{\tau=0}^t dx \leq  - \frac{1}{\epsilon^2} \iint_{[0,t] \times \mathbb{T}^n} \rho |u- \bar{u}| ^2 dxd\tau 
\nonumber\\
&  - \frac{1}{\epsilon} \iint_{[0,t] \times \mathbb{T}^n} \rho \nabla \bar{u} :[(u-\bar{u}) \otimes (u-\bar{u})]dxd\tau 
 - \frac{1}{\epsilon} \int \int_{[0,t] \times \mathbb{T}^n} \rho \nabla \bar{v} : [ (v-\bar v) \otimes ( u - \bar u)] dxd\tau 
 \nonumber\\
 & + \iint_{[0,t] \times \mathbb{T}^n} \bar{e}\cdot\frac{\rho}{\bar{\rho}}(\bar{u} - u)dxd\tau - \frac{1}{\epsilon} \iint_{[0,t] \times \mathbb{T}^n} p(\rho|\bar{\rho}) \dive \bar{u} dx d\tau 
 \nonumber\\ 
 & - \frac{1}{\epsilon} \iint_{[0,t] \times \mathbb{T}^n}\big[ \mu(\rho)((v-\bar{v})\cdot  \nabla \dive \bar{u} - (u-\bar{u})\cdot\nabla \dive \bar{v}) + \nabla \mu(\rho)\cdot(\nabla \bar{u}(v-\bar{v}) - \nabla \bar{v}(u-\bar{u}))\big]dxd\tau
 \nonumber\\ 
 & - \frac{1}{2\epsilon} \iint_{[0,t]\times \mathbb{T}^n}  \big[ \nabla \lambda(\rho)\cdot((v-\bar{v}) \dive \bar{u} - (u-\bar{u}) \dive \bar{v}) 
 +   \lambda(\rho)((v-\bar{v})\cdot \nabla \dive \bar{u} - (u-\bar{u})\cdot \nabla \dive \bar{v})\big ]dxd\tau  
 \nonumber\\ 
 & + \frac{1}{\epsilon} \iint_{[0,t] \times \mathbb{T}^n} \frac{\rho}{\bar{\rho}} \Bigg [(\mu(\bar{\rho}) \dive \nabla \bar{v} + {}^t\nabla \mu(\bar{\rho}) {}^t\nabla \bar{v})\cdot(\bar{u} - u)  +
  \frac{1}{2}( \nabla \lambda(\bar{\rho}) \dive \bar{v} + \lambda(\bar{\rho}) \nabla \dive \bar{v} )\cdot(\bar{u} - u)\Bigg ]dxd\tau 
 \nonumber\\ 
 & - \frac{1}{\epsilon} \iint_{[0,t] \times \mathbb{T}^n} \frac{\rho}{\bar{\rho}} \Bigg [ (\mu(\bar{\rho}) \dive {}^t\nabla \bar{u} + {}^t\nabla \mu(\bar{\rho}) \nabla \bar{u})\cdot(\bar{v}-v) 
  + \frac{1}{2} (\nabla \lambda(\bar{\rho}) \dive \bar{u} +  \lambda(\bar{\rho}) \nabla \dive \bar{u} )\cdot (\bar{v}-v) \Bigg ]dxd\tau.
\end{align}
We recall that $\dive {}^t\nabla \bar u = \nabla \dive \bar u$ and therefore  $\dive \nabla \bar v = \nabla \dive \bar v$ being $\nabla \bar v$   symmetric.
% \begin{equation*}
% \begin{split}
% &\dive \nabla \bar v = \nabla \dive \bar v \; \; \; \text{ since \; \; ${}^t\nabla v = \nabla v$} \\
% &\dive {}^t\nabla \bar u = \nabla \dive \bar u.
% \end{split}
% \end{equation*}
Hence, we can collect terms as follows:
% Now we collect the similar terms, namely we define:
\begin{align*}
% \begin{split}
I_1 := & -\frac{1}{\epsilon} \iint_{[0,t] \times \mathbb{T}^n} \left ( \mu(\rho)- \frac{\rho}{\bar{\rho}} \mu(\bar{\rho}) \right) ( \nabla \dive \bar{u} \cdot (v-\bar{v})-  \nabla \dive\bar{v} \cdot (u-\bar{u}))dxd\tau  
\\ 
& - \frac{1}{\epsilon} \iint_{[0,t] \times \mathbb{T}^n} \left( \nabla \mu(\rho) - \frac{\rho}{\bar{\rho}} \nabla \mu(\bar{\rho}) \right)\cdot( \nabla \bar{u}(v-\bar{v}) - \nabla \bar{v} (u-\bar{u}))dxd\tau.
% \end{split}
\end{align*}
In addition, recalling also the definition of $\displaystyle{v = \frac{\nabla \mu(\rho)}{\rho}}$, we conclude:
\begin{equation*}
\begin{split}
I_1 = & -\frac{1}{\epsilon} \iint_{[0,t] \times \mathbb{T}^n} \rho \left ( \frac{\mu(\rho)}{\rho}- \frac{\mu(\bar\rho)}{\bar{\rho}}  \right) ( \nabla \dive \bar{u} \cdot (v-\bar{v})-  \nabla \dive\bar{v} \cdot (u-\bar{u}))dxd\tau  
\\ 
& - \frac{1}{\epsilon} \iint_{[0,t] \times \mathbb{T}^n}\rho ( v -  \bar{v} )\cdot( \nabla \bar{u}(v-\bar{v}) - \nabla \bar{v} (u-\bar{u}))dxd\tau.
\end{split}
\end{equation*} 
Morevoer, we define
\begin{equation*}
\begin{split}
I_2 := &- \frac{1}{2\epsilon} \iint_{[0,t] \times \mathbb{T}^n} \left(\lambda(\rho) - \frac{\rho}{\bar{\rho}} \lambda(\bar{\rho}) \right)((v-\bar{v}) \cdot \nabla \dive \bar{u} - (u-\bar{u})\cdot \nabla \dive \bar{v}) dxd\tau 
\\ 
& - \frac{1}{2\epsilon}\iint_{[0,t] \times \mathbb{T}^n} \left(\nabla \lambda(\rho) - \frac{\rho}{\bar{\rho}} \nabla \lambda(\bar{\rho}) \right)\cdot ((v-\bar{v})\dive \bar{u} - (u-\bar{u}) \dive \bar{v} )dxd\tau.
\end{split}
\end{equation*}
Since $\displaystyle{\lambda(\rho)= 2(\rho\mu'(\rho)-\mu(\rho))}$, one has
\begin{equation*}
 \begin{split}
I_2 = & - \frac{1}{2\epsilon} \iint_{[0,t] \times \mathbb{T}^n} \rho \left(\frac{\lambda(\rho)}{\rho} - \frac{\lambda(\bar{\rho})}{\bar{\rho}} \right)((v-\bar{v}) \cdot \nabla \dive \bar{u} - (u-\bar{u})\cdot \nabla \dive \bar{v}) dxd\tau 
\\ 
& - \frac{1}{\epsilon}\iint_{[0,t] \times \mathbb{T}^n} \rho (\mu''(\rho)\nabla\rho -  \mu''(\bar \rho)\nabla\bar\rho)\cdot ((v-\bar{v})\dive \bar{u} - (u-\bar{u}) \dive \bar{v})dxd\tau.
 \end{split}
\end{equation*}
Therefore 
\begin{align}\label{eq:last}
% \begin{split}
I_1 + I_2 = & - \frac{1}{\epsilon}  \iint_{[0,t] \times \mathbb{T}^n} \rho((\mu''(\rho)\nabla \rho - \mu''(\bar{\rho})\nabla \bar{\rho})\cdot((v-\bar{v})\dive \bar{u} + (\bar{u}-u)\dive \bar{v}))dxd\tau \nonumber\\ 
& - \frac{1}{\epsilon} \iint_{[0,t] \times \mathbb{T}^n} \rho (\mu'(\rho)- \mu'(\bar{\rho}))((v-\bar{v})\cdot\nabla \dive \bar{u}+(\bar{u}-u)\cdot\nabla \dive \bar{v})dxd\tau 
\nonumber\\
& - \frac{1}{\epsilon}\iint_{[0,t] \times \mathbb{T}^n} \left [\rho (v-\bar{v}) \cdot \nabla \bar{u} (v-\bar{v}) - \rho(v-\bar{v}) \nabla \bar{v} (u-\bar{u}) \right ] dxd\tau
% \end{split}
\end{align}
Finally, using \eqref{eq:last} in \eqref{corradded2} we obtain \eqref{eq:relen} and the proof is complete.
\end{proof}
\section{Stability result and convergence of the diffusive limit}
\label{sec:stabconv}
With the relative entropy estimate \eqref{eq:relen} of Theorem \ref{relativentropy} at hand, we are now able to control our diffusive relaxation limit in terms of the quantity
\begin{equation}\label{eq:distfi}
    \Psi(t) := \int_{\mathbb{T}^n} \left (h(\rho|\bar{\rho}) + \frac{1}{2} \rho \left| \frac{m}{\rho} - \frac{\bar{m}}{\bar{\rho}}\right|^2 + \frac{1}{2} \rho \left|\frac{J}{\rho} - \frac{\bar{J}}{\bar{\rho}}\right|^2 \right )dx.
\end{equation}
The proof of our convergence result will follow the blueprint of \cite{LT,LT2}, in particular generalizing the results of the latter to our more general case in terms of the capillarity coefficient, thanks to the enlarged reformulation of the system due to  \cite{Bresh}. 
To this end, let us first remark that, since we are dealing here to $\gamma$--law gases, $\gamma>1$, we have
\begin{equation*}
     h(\rho) = \frac{1}{\gamma-1}\rho^\gamma.
 \end{equation*}
Therefore 
\begin{equation}\label{pcontrol}
    p(\rho|\bar{\rho})= (\gamma - 1) h(\rho|\bar{\rho}),
\end{equation}
and the error term in \eqref{eq:relen} involving the pressure will be then controlled in terms of the relative entropy, namely in terms of the ``distance'' $\Psi$ defined in \eqref{eq:distfi}.
It is worth observing that the same kind of control can be obtained for general monotone pressure laws, with $h$ given as in \eqref{eq:defh} and satisfying appropriate conditions, and for positive densities; see \cite{LT,GLT,LT2} for details, as well as for discussions about the metric induced by \eqref{eq:distfi}.
Moreover, to control the last two terms of \eqref{eq:relen}, 
we take advantage of the results contained in \cite{Bresh}, an in particular the followig one, that we report here below for the sake of completeness.
\begin{lemma}\cite[Lemma 35]{Bresh}\label{lemma8}
Let assume $\mu(\rho)= \rho^{\frac{s+3}{2}}$ with $\gamma \geq s+2$ and $s \geq -1$. We have $$\rho |\mu'(\rho)-\mu'(\bar{\rho})|^2 \leq C(\bar{\rho})h(\rho|\bar{\rho}),$$ with $C(\bar{\rho})$ uniformly bounded for $\bar{\rho}$ belonging to compact sets in $\mathbb{R}^+ \times \mathbb{T}^n$.
\end{lemma}

We are now ready to state our main convergence theorem.
\begin{theorem}\label{STABILITY}
Let $T>0$ be fixed and let $(\rho,m, J)$ be as in Definition \ref{ws} and $\bar\rho$   be a smooth solution of $\eqref{gf}$ with $\bar{\rho} \geq \delta > 0$, and define $\bar m$ and $\bar J$ by \eqref{eq:defbarutheo}. Assume the pressure $p(\rho)$ is given by the $\gamma$--law $\rho^{\gamma}$, $\gamma > 1$, and assume $\mu(\rho)= \rho^{\frac{s+3}{2}}$ with $\gamma \geq s+2$ and $s \geq -1$. 
% \begin{enumerate}
% \item[1)] $\rho h''(\rho) = p'(\rho)$ with $p'(\rho)> 0$,
% \item[2)] $p(\rho) = \rho^{\gamma}$ with $\gamma > 1$.
% \end{enumerate}
Then, for any $t \in [0,T]$,  the stability estimate
\begin{equation}\label{stabtheosec4}
\Psi(t) \leq C(\Psi(0) + \epsilon^4),
\end{equation}
holds true, where $C$ is a positive constant depending on $T$, $M$, the $L^1$ bound for $\rho$, assumed to be uniform in $\epsilon$, $\bar\rho$,  and its derivatives. Moreover, if $\Psi(0) \rightarrow 0$ as $\epsilon \rightarrow 0$, then as $\epsilon \rightarrow 0$
\begin{equation*}
\sup_{t \in [0,T]} \Psi(t) \rightarrow 0.
\end{equation*}
\end{theorem}
\begin{proof}
In view of the definition of $\Psi$ in  \eqref{eq:distfi}, from the relative entropy estimate given by   Theorem \ref{relativentropy} we get:
\begin{align}\label{phie}
& \Psi(t) + \frac{1}{\epsilon^2} \iint_{[0,t]\times \mathbb{T}^n} \rho \left| \frac{m}{\rho} - \frac{\bar{m}}{\bar{\rho}} \right|^2dxd\tau \leq \Psi(0) + \iint_{[0,t]\times \mathbb{T}^n} (|Q|+ |E|) dxd\tau,
\end{align}
 where the terms $Q$ and $E$ are given by
\begin{equation*}
E : = \bar{e} \cdot \frac{\rho}{\bar{\rho}} \left( \frac{m}{\rho} - \frac{\bar{m}}{\bar{\rho}} \right), \ Q = Q_1 + Q_2,
\end{equation*}
with
\begin{align*}
Q_1 : = & - \frac{1}{\epsilon} \iint_{[0,t]\times \mathbb{T}^n} \rho \nabla \bar{u} : [(u-\bar{u}) \otimes (u-\bar{u})]dxd\tau 
\\
& - \frac{1}{\epsilon} \iint_{[0,t] \times \mathbb{T}^n} \rho \nabla \bar{u} : [(v-\bar{v}) \otimes (v-\bar{v})]dxd\tau 
  - \frac{1}{\epsilon} \iint_{[0,t]\times \mathbb{T}^n} p(\rho|\bar{\rho}) \dive \bar{u} dxd\tau 
\end{align*}
\begin{align*}
Q_2 : = & -  \frac{1}{\epsilon} \iint_{[0,t]\times \mathbb{T}^n} \rho[(\mu''(\rho)\nabla \rho - \mu''(\bar{\rho})\nabla \bar{\rho})\cdot((v-\bar{v})\dive \bar{u} - (u-\bar{u})\dive \bar{v})]dxdt 
\\ 
&  - \frac{1}{\epsilon} \iint_{[0,t]\times \mathbb{T}^n} \rho (\mu'(\rho)- \mu'(\bar{\rho}))((v-\bar{v}) \cdot \nabla \dive \bar{u} -  (u-\bar{u})\cdot \nabla \dive \bar{v})dxd\tau.
\end{align*}
We use the Young inequality and the previous results to estimate $E$ and $Q_1$ (as in \cite{LT,LT2}) and $Q_2$ (following \cite{Bresh}) in terms of the relative entropy itself. We start from the error term $E$:
\begin{align*}
\iint_{[0,t] \times \mathbb{T}^n} \left|E\right|dxd\tau & \leq \frac{\epsilon^2}{2}
\iint_{[0,t] \times \mathbb{T}^n} \left| \frac{\bar{e}}{\bar{\rho}} \right|^2 \rho dxd\tau + \frac{1}{2\epsilon^2}\iint_{[0,t] \times \mathbb{T}^n} \rho \left| \frac{m}{\rho} - \frac{\bar{m}}{\bar{\rho}} \right|^2dxd\tau
\\ 
& \leq CT\epsilon^4 + \frac{1}{4\epsilon^2}\iint_{[0,t] \times \mathbb{T}^n} \rho \left| \frac{m}{\rho} - \frac{\bar{m}}{\bar{\rho}} \right|^2dxd\tau,
\end{align*}
using the bounds for $\bar\rho$, the $L^1$ bound for $\rho$, and in view of the fact that, as shown in 
 \eqref{eq:error}, the error term $\bar e$ is $O(\epsilon)$.
For the term  $Q_1$ we use again the  the fact that ${\nabla\bar{u}}=O(\epsilon)$ to conclude 
\begin{align*}
 \frac{1}{\epsilon} \iint_{[0,t] \times \mathbb{T}^n} \rho \nabla \bar{u} : [(u-\bar{u}) \otimes (u-\bar{u})]dxd\tau \leq C_1 \iint_{[0,t] \times \mathbb{T}^n} \rho \left| \frac{m}{\rho} - \frac{\bar{m}}{\bar{\rho}} \right|^2dxd\tau,
\end{align*}
\begin{align*}
\frac{1}{\epsilon}\iint_{[0,t] \times \mathbb{T}^n} \rho \nabla \bar{u} : [(v-\bar{v}) \otimes (v-\bar{v})]dxd\tau & \leq C_2 \iint_{[0,t] \times \mathbb{T}^n} \rho \left|\frac{J}{\bar{\rho}} - \frac{\bar{J}}{\bar{\rho}} \right|^2dxd\tau,
\end{align*}
\begin{align*}
\frac{1}{\epsilon}\iint_{[0,t] \times \mathbb{T}^n} p(\rho|\bar{\rho})\dive \bar{u} dxd\tau \leq C_3 \iint_{[0,t] \times \mathbb{T}^n} h(\rho|\bar{\rho})dxd\tau,
\end{align*}
the latter thanks to \eqref{pcontrol} as well.
For the new term $Q_2$ coming from the   formulation of the relative entropy estimate  of \cite{Bresh}, the strategy is the same: we shall take advantage of the estimates from that paper, by carefully taking into account of the singular coefficient in terms of the relaxation parameter $\epsilon$. For the first term  we define
\begin{align*}
& \frac{1}{\epsilon} \iint_{[0,t] \times \mathbb{T}^n} \rho(\mu''(\rho) \nabla \rho - \mu''(\bar{\rho})\nabla \bar{\rho})\cdot((v-\bar{v})\dive \bar{u} + (\bar{u}-u)\dive \bar{v})dxd\tau 
\\
& = Q_{21} + Q_{22}, 
\end{align*}
where
\begin{align*}
& Q_{21} :=  \frac{1}{\epsilon} \iint_{[0,t] \times \mathbb{T}^n} \sqrt{\rho}(\mu''(\rho) \nabla \rho - \mu''(\bar{\rho})\nabla \bar{\rho})\cdot \sqrt{\rho}(v-\bar{v})\dive \bar{u} dxd\tau, \\
& Q_{22} := \frac{1}{\epsilon} \iint_{[0,t] \times \mathbb{T}^n} \sqrt{\rho}(\mu''(\rho) \nabla \rho - \mu''(\bar{\rho})\nabla \bar{\rho})\cdot\sqrt{\rho}(\bar u -u)\dive \bar{v}dxd\tau .
\end{align*}
Again, ${\dive \bar u }=O(\epsilon)$ and, since $\mu''(\rho)\nabla \rho - \mu''(\bar \rho)\nabla \bar \rho = \frac{s+1}{2}(v- \bar v) $, we readily  obtain 
\begin{align*}
 &   Q_{21} \leq C_4 \iint_{[0,t] \times \mathbb{T}^n} \rho \left|\frac{J}{\rho}- \frac{\bar J}{\bar \rho}\right|^2 dxd\tau, \\
 & Q_{22} \leq C_5 \iint_{[0,t] \times \mathbb{T}^n} \rho \left|\frac{J}{\rho}-\frac{\bar J}{\bar \rho}\right|^2 dxd\tau +  \frac{1}{4 \epsilon^2}\int \int_{[0,t] \times \mathbb{T}^n} \rho \left| \frac{m}{\rho} - \frac{\bar{m}}{\bar{\rho}} \right|^2dxd\tau,
\end{align*}
using  Young's inequality for the second estimate. 
Analogously, we split the second term in $Q_2$ in two:
\begin{align*}
& \frac{1}{\epsilon} \iint_{[0,t] \times \mathbb{T}^n} \rho(\mu'(\rho)-\mu'(\bar{\rho}))((v-\bar{v})\cdot\nabla \dive \bar{u}+ (\bar{u}-u)\cdot\nabla \dive \bar{v})dxd\tau  
\\
& = \frac{1}{\epsilon} \iint_{[0,t] \times \mathbb{T}^n} \sqrt{\rho}(\mu'(\rho)-\mu'(\bar{\rho}))\sqrt{\rho}(v-\bar{v})\cdot\nabla \dive \bar{u}dxd\tau  \\
& + \frac{1}{\epsilon} \iint_{[0,t] \times \mathbb{T}^n} \sqrt{\rho}(\mu'(\rho)-\mu'(\bar{\rho}))\sqrt{\rho}(\bar u -u )\cdot\nabla \dive \bar{v}dxd\tau.
\end{align*}
Hence, we use  Young's inequality and Lemma $\ref{lemma8}$ to bound the first term in view of   ${\nabla \dive \bar u } =O(\epsilon)$, while for the second one we take advantage of the control given by the friction term: 
\begin{align*}
& \frac{1}{\epsilon} \iint_{[0,t] \times \mathbb{T}^n} \rho(\mu'(\rho)-\mu'(\bar{\rho}))((v-\bar{v})\cdot\nabla \dive \bar{u}+ (\bar{u}-u)\cdot\nabla \dive \bar{v})dxd\tau
\\ 
& \leq C_6 \iint_{[0,t] \times \mathbb{T}^n} \left ( h(\rho|\bar{\rho}) + \rho \left|\frac{J}{\rho} - \frac{\bar{J}}{\bar{\rho}}\right|^2\right )dxd\tau + \frac{1}{8\epsilon^2} \iint_{[0,t] \times \mathbb{T}^n} \rho \left|\frac{m}{\rho}-\frac{\bar{m}}{\bar{\rho}}\right|^2 dxd\tau.
\end{align*}
Finally the relative entropy inequality becomes:
\begin{align*} 
\Psi(t) + \frac{1}{2\epsilon^2} \int \int_{[0,t] \times \mathbb{T}^n} \rho \left| \frac{m}{\rho} - \frac{\bar{m}}{\bar{\rho}} \right|^2dxd\tau \leq \Psi(0) + \tilde{C}\epsilon^4 + C \int_0^t \Psi(\tau)d\tau,
\end{align*}
and the Gronwall's Lemma gives the desired result.
\end{proof}

\section{The high friction limit of Navier--Stokes--Korteweg system}\label{sec:NSK}
In the same spirit of the previous discussions, in this section we want to study the high-friction limit in the case of  the Navier--Stokes--Korteweg system, which, in the (enlarged) formulation and after the scaling described above, rewrites as follows:
\begin{equation}\label{NSK}
	\left\{\begin{aligned}
& \partial_t \rho + \frac{1}{\epsilon} \dive m  = 0
\\
& \partial_t m  + \frac{1}{\epsilon} \dive \left( \frac{m \otimes m}{\rho} \right) + \frac{1}{\epsilon} \nabla p(\rho) - \frac{2 \nu}{\epsilon} \dive(\mu_L(\rho)Du) - \frac{\nu}{\epsilon} \nabla(\lambda_L(\rho) \dive u) \\
& \ = \frac{1}{\epsilon} \dive S_1  - \frac{1}{\epsilon^2}m
\\
&\partial_t J + \frac{1}{\epsilon} \dive\left(\frac{J \otimes m}{\rho} \right) + \frac{1}{\epsilon} \dive S_2= 0. 
	\end{aligned}\right.
\end{equation}
In system \eqref{NSK},  $m=\rho u$ , $J=\rho v $, the viscosity coefficient $\nu$ is positive, and, as denoted above, 
$$Du= \frac{ \nabla u + {}^t\nabla u}{2}$$ 
is the symmetric part of the gradient $\nabla u$ and we recall that the Lam\'e coefficient verifies 
\begin{equation}
    \label{lame2}
    \mu_L(\rho)\geq 0;\ \frac2n \mu_L(\rho) +\lambda_L(\rho) \geq 0.
\end{equation}
Moreover, the effective velocity $v = \nabla \mu(\rho)/\rho$ and the stresses $S_1$ and $S_2$ are the same  of the Euler--Korteweg system, namely
$$ \dive S_1 = \dive (\mu(\rho)\nabla v) + \frac{1}{2 \epsilon} \nabla(\lambda(\rho)\dive v)$$
and 
$$\dive S_2 = (\mu(\rho) {}^t\nabla u) + \frac{1}{2 \epsilon} \nabla(\lambda(\rho)\dive u),$$ 
 as well as the definition of the functions $\mu(\rho)$ and $\lambda(\rho)$. 
 As already pointed out above, we stress once again that the coefficients $\mu_L(\rho$)and $\lambda_L(\rho)$ need not to coincide with $\mu(\rho)$ and $\lambda(\rho)$, and we shall only assume their $L^1$ norm is  bounded uniformly in $\epsilon$, which can be   viewed as   a control of them in terms of the pressure term $\rho^\gamma$ and using the energy bound $E_o$.
 
 The Hilbert expansion applied to system $\eqref{NSK}$ will give us the same formal limit of the previous case, that is the  viscosity term will affect the expansion only for higher terms, and therefore,  the limit solution $\bar\rho$ as $\epsilon \rightarrow 0$ satisfies the following equation: 
\begin{equation}\label{GFfor NS}
\partial_t \bar{\rho} + \dive(- \nabla p(\bar{\rho}) + \dive S_1(\bar{\rho})) = 0,
\end{equation}
 while the (nonzero) leading term for the momentum is given by 
 \begin{equation}
     \label{eq:equmom2}
      \bar m = \epsilon(-\nabla p(\rho) + \dive \bar S_1).
 \end{equation}
 Indeed,  we introduce the asymptotic expansion of the state variables:
\begin{align*}
&\rho = \rho_0 + \epsilon\rho_1 + \epsilon^2 \rho_2 + \cdots\\
&m = m_{0} + \epsilon m_{1} + \epsilon^2 m_{2} + \cdots
\end{align*}
in the system $\eqref{NSK}$ and collect the terms of the same order; the expansion for $J$ will clearly come from the one of $\rho$. 
Then, from the mass conservation we get:
\begin{align*}
&O(\epsilon^{-1}):  & & \dive m_{0} = 0; \\ 
&O(1): & & \partial_t \rho_0 + \dive m_{1} = 0; \\
&O(\epsilon): & &  \partial_t \rho_1 + \dive m_{2} =0;\\
&O(\epsilon^2): & &  \dots\\
\end{align*}
while, from the momentum equation we get:
\begin{align*}
&O(\epsilon^{-2}) :  & &    m_{0}=0; \\ 
&O(\epsilon^{-1}) :  & &    -m_{1} = \nabla p(\rho_0) - \dive S_1(\rho_0); 
\\
&O(1): & &  -m_{2} = \nabla( p'(\rho_0)\rho_1) - \dive(\mu'(\rho_0)\rho_1 \nabla v_0 + \mu(\rho)\nabla v_1) 
\\
& & &   \qquad \qquad + \nabla(\lambda'(\rho_0)\rho_1 \dive v_0 + \lambda(\rho_0) \dive v_1) 
\\
& & & \qquad\qquad - 2 \nu\dive\left(\mu_L(\rho_0)D\left(\frac{m_1}{\rho_0}\right)\right) - \nu \nabla\left(\lambda_L(\rho)\dive \frac{m_1}{\rho_0}\right); \\
&O(\epsilon): & &  \dots\\
\end{align*}
Hence, from these first relations,  we recover the equilibrium relation $m_{0} = 0$, the Darcy's law $m_{1} = - \nabla_x p(\rho_0) + \dive_x S_1(\rho_0)$, and 
the  gradient flow dynamic $\eqref{GFfor NS}$ for the leading term $\rho_0$.

In the same spirit of Section 2, we rewrite the scalar equation $\eqref{GFfor NS}$ in the same form of the ``hyperbolic part'' of system \eqref{NSK} by adding an appropriate error term. To this end, 
let us consider   $\bar\rho$   a smooth solution of \eqref{GFfor NS} and assume $\bar m$ is given by \eqref{eq:equmom2} and, as said before, $\bar J = \nabla \mu(\bar\rho)$. Then $(\bar \rho , \bar m, \bar J)$  satisfies
\begin{equation}\label{ns-scaledstrong}
	\left\{\begin{aligned}
& \partial_t  \bar \rho + \frac{1}{\epsilon} \dive\bar m  = 0 \\
& \partial_t\bar m   + \frac{1}{\epsilon} \dive \left(\frac{\bar m \otimes \bar m }{\bar\rho} \right) + \frac{1}{\epsilon} \nabla p( \bar \rho) = \frac{1}{\epsilon} \dive \bar S_1  - \frac{1}{\epsilon^2} \bar m + \bar e\\
&\partial_(\bar J + \frac{1}{\epsilon} \dive \left(\frac{\bar J \otimes \bar m}{\bar \rho} \right) + \frac{1}{\epsilon} \dive \bar S_2 = 0,
	\end{aligned}\right.
\end{equation}
where 
$$\bar e = \frac{1}{\epsilon} \dive \left(\frac{\bar m \otimes \bar m }{\bar \rho} \right) + \bar m_t = O(\epsilon).
$$
The idea of introducing the system 
 $\eqref{ns-scaledstrong}$ is that of recostructing the same first--order part of the relaxing system, to take advantage of the properties which link the entropy and the convective terms, and then obtain in a more direct way the relative energy estimate, as already done in Section \ref{sec:relenes}, and therefore there is no need to introduce viscosity terms (and thus extra errors) in this reformulation of the equilibrium dynamics.  
 As consequence, 
%  does no appear the viscosity terms. Indeed the only non trivial $O(\epsilon^{-1})$ terms in the Hilbert expansion of the momentum equation, satisfy the following equation $m_1 = -\nabla p(\rho_0) + \dive \bar S_1(\rho_0)$, the same as Euler-Korteweg in Section 2. Thefore in order to compare weak solution of $\eqref{NSK}$ and strong solution of $\eqref{ns-scaledstrong}$ we can neglect the viscous forces for the limit smooth solution. 
 the structure of the two systems is the same of the one considered above, and hence  we shall emphasize here below only the differences with respect to the previous calculations in obtaining the desired relative entropy inequality.
 
%  that turns out to take into account the viscous terms only for $\eqref{NSK}$ and not for $\eqref{ns-scaledstrong}$. We first derive the usual energy estimates associated to smooth solutions of $\eqref{NSK}$ and $\eqref{ns-scaledstrong}$ and then we pass to the weak formulation for $(\rho,m,J)$. Then we derive a relative entropy inequality that will be used in the end to obtain a stability result via Gronwall integral inequality. 
% \newline
Let us start by recalling the  constitutive relations for the functions involved in $\eqref{NSK}$, that is the
$\gamma$--law pressure $p(\rho) = \rho^{\gamma}$, $\mu(\rho)= \rho^{\frac{s+3}{2}}$ with the conditions $\gamma > 1$, $s+2 \leq \gamma$, $s \geq -1$ and  $\lambda(\rho)=2( \rho \mu'(\rho) - \mu(\rho))$. 
The mechanical energy associated to $\eqref{NSK}$ is given by
\begin{equation*}
\eta(\rho,m,J) = \frac{1}{2} \frac{|m|^2}{\rho} + \frac{1}{2} \frac{|J|^2}{\rho} + h(\rho),
\end{equation*}
and, proceeding as in the previous sections, we (formally) obtain 
% we multiply the momentum equation for $m$ by $u$ and the momentum equation for $J$ by $v$, then using the continuity and integrating over $\mathbb{T}^n$ we get:
\begin{equation*}
 \frac{d}{dt} \int_{\mathbb{T}^n} \eta(\rho,m,J) dx + \frac{2\nu}{\epsilon} \int_{\mathbb{T}^n}\mu_L(\rho)|D(u)|^2 dx + \frac{\nu}{\epsilon} \int_{\mathbb{T}^n} \lambda_L(\rho)|\dive u|^2 dx = 
- \frac{1}{\epsilon^2} \int_{\mathbb{T}^n} \frac{|m|^2}{\rho}dx.
\end{equation*}
In particular, as it is well known, condition \eqref{lame2} implies the mechanical energy dissipates along solutions of $\eqref{NSK}$.
On the other hand, the entropy $\bar \eta(\bar \rho, \bar m , \bar J)$ associated to $\eqref{ns-scaledstrong}$ satisfies:
\begin{equation}\label{etabarNS}
\frac{d}{dt} \int_{\mathbb{T}^n} \bar \eta(\bar \rho, \bar m , \bar J) dx = - \frac{1}{\epsilon^2} \int_{\mathbb{T}^n} \frac{|\bar m|^2}{\rho}dx + \int_{\mathbb{T}^n} \bar e \frac{\bar m }{\bar \rho}dx.
\end{equation}
We state here below the definition of weak solutions we shall consider in the study of our relaxation limit
\begin{definition}\label{deFNS}
 ($\rho$, $m$, $J$) with $\rho \in C([0, \infty);(L^1(\mathbb{T}^n))$ $(m,J) \in C([0, \infty);(L^1(\mathbb{T}^n))^{2n})$, $\rho \geq 0$, is a weak (periodic) solution of $\eqref{NSK}$ if 
\begin{align*}
    & \sqrt{\rho}u,  \sqrt{\rho}v \in L^{\infty}((0,T);L^2(\mathbb{T}^n)^n),\ \rho \in C([0, \infty);(L^\gamma(\mathbb{T}^n)),
    \\
    & \mu_L(\rho)D(u) \in L^1((0,T);L^1(\mathbb{T}^n)^{2n} ),
    \ \lambda_L(\rho)\dive u \in L^1((0,T);L^1(\mathbb{T}^n))
\end{align*}
% \; \frac{\mu(\rho)}{\sqrt{\rho} } \in L^{\infty}((0,T);L^2(\mathbb{T}^n) ) 
and $(\rho, m,J)$ satisfy for all $\psi \in C^1_c([0, \infty); C^1(\mathbb{T}^n))$ and for all $\phi, \varphi \in C^1_c([0, \infty); C^1(\mathbb{T}^n)^n)$: 
\begin{align*}
& - \iint_{(0,+\infty)\times\mathbb{T}^n}\Bigg (\rho \psi_t + \frac{1}{\epsilon} m \cdot \nabla_x \psi \Bigg )dxdt = \int_{\mathbb{T}^n} \rho(x,0)\psi(x,0);\\
&  -  \iint_{(0,+\infty)\times\mathbb{T}^n} \Bigg[m \cdot (\phi)_t + \frac{1}{\epsilon}\left(\frac{m \otimes m}{\rho} : \nabla_x \phi \right) + \frac{1}{\epsilon} p(\rho) \dive \phi  -   \frac{2\nu}{\epsilon} \mu_L(\rho) D(u): \nabla \phi \\ 
& \; \; -  \frac{\nu}{\epsilon} \lambda_L(\rho)\dive u \dive \phi 
+ \frac{1}{\epsilon}\left( \mu(\rho) v \cdot \nabla \dive (\phi) + \nabla \mu(\rho) \cdot (\nabla \phi v ) \right) + \\
& \; \; \; \frac{1}{\epsilon} \left( \frac{1}{2} \nabla \lambda(\rho) \cdot v \dive \phi + \frac{1}{2} \lambda(\rho) v \cdot \nabla \dive \phi \right)  \Bigg]dxdt   = \\
& \ - \frac{1}{\epsilon^2} \iint_{(0,+\infty)\times\mathbb{T}^n}m \cdot \phi dxdt + \int_{\mathbb{T}^n} m(x,0) \cdot\phi(x,0)dx,
\end{align*}
where we have used the identity 
$$\displaystyle{S= - p(\rho) \mathbb{I} + S_1= - p(\rho)\mathbb{I} + \mu(\rho)\nabla v + \frac{1}{2}\lambda(\rho)\dive v \mathbb{I}},$$ 
\begin{align*}
& - \iint_{(0,+\infty)\times\mathbb{T}^n} \Bigg[J \cdot\varphi_t + \frac{1}{\epsilon}\left(\frac{J \otimes m}{\rho} : \nabla_x \varphi \right) - \frac{1}{\epsilon}\Bigg( \mu(\rho) u \cdot   (\nabla \dive \varphi ) + \nabla \mu(\rho) \cdot (\nabla \varphi u ) \\
&\ + \frac{1}{2} \nabla \lambda(\rho) \cdot u \dive \varphi + \frac{1}{2}  \lambda(\rho) u \cdot  \nabla \dive \varphi \Bigg)  \Bigg]dxdt =  \int_{\mathbb{T}^n} J(x,0) \cdot \varphi(x,0)dx,
\end{align*}
where we have used the identity 
$$\displaystyle{S_2= \mu(\rho)^t\nabla u + \frac{1}{2}\lambda(\rho)\dive u \mathbb{I}}.$$
If in addition $ \eta (\rho,m,J) \in C([0, \infty); L^1(\mathbb{T}^n))$ and $(\rho,m,J)$ satisfy
\begin{align}\label{dissNS}
& \iint_{(0,+\infty)\times\mathbb{T}^n} \left( \eta(\rho,m,J) \right) \dot{\theta}(t) dxdt \leq  \int_{\mathbb{T}^n} \left( \eta(\rho,m,J)\right)|_{t=0} \theta(0)dx \nonumber\\
& \ - \frac{1}{\epsilon^2}  \iint_{(0,+\infty)\times\mathbb{T}^n} \frac{|m|^2}{\rho} \theta(t) dxdt - \frac{1}{\epsilon}\iint_{(0,\infty) \times \mathbb{T}^n} \mu_L(\rho)|D(u)|^2\theta(t) dxdt \nonumber \\
& \ - \frac{1}{\epsilon}\iint_{(0,\infty) \times \mathbb{T}^n} \lambda_L(\rho)|\dive u|^2 \theta(t)dxdt
\end{align}
for any non-negative $\theta \in W^{1,\infty}[0, \infty)$ compactly supported on $[0,\infty)$, then $(\rho,m,J)$ is called a \emph{dissipative} weak solution.

If $\eta(\rho,m,J) \in C([0,\infty);L^1(\mathbb{T}^n))$ and $(\rho,m,J)$ satisfy $\eqref{dissNS}$ as an equality, then $(\rho,m,J)$ is called a \emph{conservative} weak solution.

We say that a dissipative (or conservative) weak (periodic) solution $(\rho,m,J)$ of $\eqref{NSK}$ with $\rho \geq 0$ has finite total mass and energy if
$$ \sup_{t \in (0,T)} \int_{\mathbb{T}^n} \rho dx \leq M <+ \infty,$$
and
$$ \sup_{t \in (0,T)} \int_{\mathbb{T}^n}  \eta(\rho,m,J) dx \leq E_o <+ \infty.$$
\end{definition}

%and for all vector valued functions $\phi,\tilde{\phi}$ such that every component belongs to $C^1([0,\infty); C^2(\mathbb{T}^n))$ :
The relative entropy calculation is contained in the next theorem. 
\begin{theorem}\label{relativeentropyNS}
Let $(\rho,m,J)$ be a dissipative (or conservative) weak solution of $\eqref{NSK}$ with finite total mass and energy according to Definition \ref{deFNS}, and $\bar \rho$ smooth solution of $\eqref{GFfor NS}$. Then
\begin{align}\label{eq:relenNS}
&\int_{\mathbb{T}^n} \eta(\rho,m,J| \bar{\rho}, \bar{m}, \bar{J})(t)dx  \leq  \int_{\mathbb{T}^n} \eta(\rho,m,J| \bar{\rho}, \bar{m}, \bar{J})(0)dx \nonumber \\
&  - \frac{2\nu}{\epsilon}\iint_{ (0,t) \times \mathbb{T}^n } \mu_L(\rho)|D(u-\bar u)|^2 dxd\tau - \frac{\nu}{\epsilon}\iint_{ (0,t) \times \mathbb{T}^n }\lambda_L(\rho)|\dive(u-\bar u )|^2dxd\tau  \nonumber\\
& - \frac{1}{\epsilon^2} \iint_{(0,t) \times \mathbb{T}^n}  \rho |u-\bar{u}|^2 dxd\tau - \frac{1}{\epsilon} \iint_{(0,t) \times \mathbb{T}^n} \rho \nabla \bar{u}: (u-\bar{u}) \otimes (u-\bar{u})dxdt \nonumber\\ 
& - \iint_{(0,t) \times \mathbb{T}^n} e(\bar{\rho},\bar{m}) \cdot \frac{\rho}{\bar{\rho}} (u-\bar{u})dxd\tau - \frac{1}{\epsilon} \iint_{(0,t) \times \mathbb{T}^n}p(\rho|\bar{\rho}) \dive \bar{u} dxd\tau \nonumber \\ 
& - \frac{1}{\epsilon} \iint_{(0,t) \times \mathbb{T}^n} \rho \; \nabla \bar{u}: (v-\bar{v}) \otimes (v-\bar{v}) dxd\tau \nonumber\\ 
& - \frac{1}{\epsilon} \iint_{(0,t) \times \mathbb{T}^n} \rho[(\mu''(\rho)\nabla \rho - \mu''(\bar{\rho})\nabla \bar{\rho})\cdot((v-\bar{v})\dive \bar{u} - (u-\bar{u})\dive \bar{v})]dxd\tau \nonumber\\ 
& - \frac{1}{\epsilon}\iint_{(0,t) \times \mathbb{T}^n} \rho (\mu'(\rho)- \mu'(\bar{\rho}))[(v-\bar{v})\cdot\nabla \dive \bar{u}-(u- \bar{u})\cdot\nabla \dive \bar{v}]dxd\tau.\nonumber\\
&  - \frac{2\nu}{\epsilon}\iint_{ (0,t) \times \mathbb{T}^n } \mu_L(\rho) D(\bar u):D(u-\bar u) dxd\tau - \frac{\nu}{\epsilon}\iint_{ (0,t) \times \mathbb{T}^n }\lambda_L(\rho)\dive \bar u(\dive u - \dive \bar u)dxd\tau
\end{align}
where 
\begin{equation*}\bar{m} = \bar{\rho}\bar{u}= \epsilon \left(- \nabla p(\bar{\rho})+   \dive S_1(\bar{\rho})\right); \ \bar{J}= \bar{\rho}\bar{v} = \nabla \mu(\bar{\rho}).
\end{equation*}
are defined as in $\eqref{eq:defbarutheo}$
\end{theorem}
\begin{proof}
To prove the relation $\eqref{eq:relenNS}$ we underline here only the differences coming from the presence of viscosity term in the momentum equation $m$. To this end, we recall that from energy inequality $\eqref{dissNS}$, using the test function  $\theta(\tau)$:
\begin{equation*}
\theta(\tau)=	
\begin{cases}
1,    & \hbox{ for } 0 \leq \tau < t, \\
\frac{t-\tau}{\mu} + 1, & \hbox{ for } t \leq \tau < t+\tau, \\
0, & \hbox{ for } \tau \geq t +\mu,
\end{cases}
\end{equation*}
as $\mu \rightarrow 0$,  one has:
\begin{align*}
\int_{\mathbb{T}^n} (\eta(\rho,m,J))|_{\tau =0}^t & \leq - \frac{1}{\epsilon^2} \iint_{(0,t)\times\mathbb{T}^n} \frac{|m|^2}{\rho} dxd\tau  - \frac{2\nu}{\epsilon}\iint_{(0,t) \times \mathbb{T}^n} \mu_L(\rho)|D(u)|^2dxd\tau \\
& - \frac{\nu}{\epsilon}\iint_{(0,t) \times \mathbb{T}^n} \lambda_L(\rho)|\dive u|^2 dxd\tau.
\end{align*}
On the other hand,  integrating over $(0,t)$ the relation $\eqref{etabarNS}$ we get:
\begin{equation*}
\int_{\mathbb{T}^n} \bar \eta(\bar \rho, \bar m , \bar J)|_{\tau = 0}^t dx = - \frac{1}{\epsilon^2} \iint_{ (0,t) \times \mathbb{T}^n } \frac{|\bar m|^2}{\rho}dx + \iint_{(0,t) \times \mathbb{T}^n} \bar e \frac{\bar m }{\bar \rho}dx.
\end{equation*}
To control the linear correction of the entropy we choose, as in Theorem \ref{relativentropy}, the following test functions in the weak formulation for the differences $(\rho - \bar\rho, m - \bar m, J-\bar J)$:
\begin{align*}
&\psi = \theta(\tau) \left( h'(\bar{\rho})- \frac{1}{2} \frac{|\bar{m}|^2}{\bar{\rho^2}} - \frac{|\bar{J}|^2}{ \bar{\rho^2}} \right) \text{ \; and} \\
& \Phi = (\phi, \varphi)= \theta(\tau) \left( \frac{\bar{m}}{\bar{\rho}}, \frac{\bar{J}}{\bar{\rho}} \right), 
\end{align*}
 where $\theta(\tau)$ is defined above. Since $D(u):\nabla \phi = D(u):D(\phi)$ and the equation for  $\bar m$ does not involve viscosity terms, 
% \textcolor{red}{We recall that: 
% $$D(u):\nabla \phi = D(u):D(\phi)$$ for any $\phi \in \in C^1_c([0, \infty); C^1(\mathbb{T}^n)^n),$ indeed: $$D(u):\nabla \phi= D(u):D(\phi) + D(u):A(\phi) = D(u):D(\phi)$$ since $$D(u):A(\phi)= \frac{1}{4}\left((\nabla u + {}^t \nabla u):( \nabla \phi - {}^t \nabla \phi) \right)=$$
% $$ \frac{1}{4}\left( \nabla u: \nabla \phi - \nabla u : {}^t \nabla \phi + {}^t \nabla u : \nabla \phi - {}^t\nabla u : {}^t\nabla \phi \right)$$ and $${}^t\nabla u : {}^t\nabla \phi = \nabla u: \nabla \phi, \; \nabla u : {}^t \nabla \phi = {}^t \nabla u : \nabla \phi.$$
the new terms due to the viscosity in  the weak formulation of the equation for $m-\bar m$ re given solely by:
\begin{align*}
  +\frac{\nu}{\epsilon} \iint_{[0,t] \times \mathbb{T}^n} \big (2  \mu_L(\rho)D(u):D(\bar u) dxd\tau +  \lambda_L(\rho) \dive u \dive \bar u \big ) dxd\tau.
\end{align*}
% since in the weak formulation for the strong solution $(\bar \rho, \bar m, \bar J)$ are absent the viscous terms $2\nu \dive(\mu(\rho)Du)$ and $\nu \nabla(\lambda(\rho) \dive u)$, indeed it is the same as Definition \ref{ws}. 
Hence, the new terms we need to handle here with respect to  Theorem \ref{relativentropy} are the following integrals:

\begin{align*}
&  \frac{2 \nu}{\epsilon} \iint_{(0,t) \times \mathbb{T}^n} \mu_L(\rho)[ D(u):D(\bar u)]dxd\tau + \frac{\nu}{\epsilon} \iint_{(0,t) \times \mathbb{T}^n} \lambda_L(\rho) \dive u \dive \bar u dx\tau\\
&  -\frac{2\nu}{\epsilon} \iint_{ (0,t) \times \mathbb{T}^n } \mu_L(\rho)|D(u)|^2 \; dxd\tau - \frac{\nu}{\epsilon} \iint_{ (0,t) \times \mathbb{T}^n } \lambda_L(\rho) |\dive u|^2 dxd\tau. 
\end{align*}
which can be rearranged as follows: 
\begin{align*}
& - \frac{2\nu}{\epsilon}\iint_{ (0,t) \times \mathbb{T}^n } \mu_L(\rho)|D(u-\bar u)|^2 dxd\tau - \frac{2\nu}{\epsilon}\iint_{ (0,t) \times \mathbb{T}^n }\mu_L(\rho) D(\bar u):D(u-\bar u) dxd\tau \\
& - \frac{\nu}{\epsilon}\iint_{ (0,t) \times \mathbb{T}^n }\lambda_L(\rho)|\dive(u-\bar u )|^2dxd\tau - \frac{\nu}{\epsilon}\iint_{ (0,t) \times \mathbb{T}^n }\lambda_L(\rho)\dive \bar u(\dive u - \dive \bar u)dxd\tau,
\end{align*} 
Hence, repeating the same calculation of Theorem \ref{relativentropy} for all remaining terms we readily obtain \eqref{eq:relenNS} and the proof is complete.
\end{proof}

Now we use Theorem \ref{relativeentropyNS} to measure the distance between the two solutions in terms of the relative entropy as in Section \ref{sec:stabconv}. To this end, we recall the definition \eqref{eq:distfi} of the ``distance'' 
$\Psi(t)$:
\begin{equation*} 
    \Psi(t) = \int_{\mathbb{T}^n} \left (h(\rho|\bar{\rho}) + \frac{1}{2} \rho \left| \frac{m}{\rho} - \frac{\bar{m}}{\bar{\rho}}\right|^2 + \frac{1}{2} \rho \left|\frac{J}{\rho} - \frac{\bar{J}}{\bar{\rho}}\right|^2 \right )dx.
\end{equation*}
Then the following theorem holds. 
\begin{theorem}\label{theo:stabNSK}
	Let $T>0$ be fixed, let $(\rho,m, J)$ be as in Definition $\ref{deFNS}$ and $\bar{\rho}$ be a smooth solution of $\eqref{GFfor NS}$ such that $\bar{\rho} \geq \delta > 0$, $\bar m$ and $\bar J$ defined as $\eqref{eq:defbarutheo}$. Assume the pressure $p(\rho)$ is given by the $\gamma$-law $\rho^{\gamma}$ with $\gamma > 1$. Assume $\mu(\rho) = \rho^{\frac{s+3}{2}}$ with $\gamma \geq s+2$ and $s \geq -1$, and 
	\begin{equation}
	    \label{lamecontrol}
	    \left\| \mu_L(\rho) \right\|_{L^{\infty}((0,t);L^1(\mathbb{T}^n))} , \left\| \lambda_L(\rho) \right\|_{L^{\infty}((0,t);L^1(\mathbb{T}^n))} \leq  \tilde E.
	\end{equation}
	for a positive constant $\tilde E$ independent from $\epsilon$.
	Then, for $t \in [0,T]$, the stability estimate 
	\begin{equation}\label{stabtheosec5}
	\Psi(t) \leq C(\Psi(0) + \epsilon^4 +\nu\epsilon),
	\end{equation}
	 holds true, where $C$ is a positive constant depending on
	$T$, $M$, the $L^1$ bound for $\rho$,  and $E_o$, the energy bound, both assumed to be uniform in $\epsilon$, $\bar{\rho}$ and its derivatives. Moreover, if $\Psi(0) \rightarrow 0$ as $\epsilon \rightarrow 0$, then as $\epsilon \rightarrow 0$
	\begin{equation}
	\sup_{t \in [0,T]} \Psi(t) \rightarrow 0.
	\end{equation}
\end{theorem}
\begin{proof}
From the definition of $\Psi(t)$  and from the relative entropy estimate given by Theorem $\ref{relativeentropyNS}$ we obtain for $t \in [0,T]$:
\begin{align*}
& \Psi(t) + \frac{1}{\epsilon^2} \iint_{[0,t]\times \mathbb{T}^n} \rho \left| \frac{m}{\rho} - \frac{\bar{m}}{\bar{\rho}} \right|^2dxd\tau + \frac{2\nu}{\epsilon} \iint_{ (0,t) \times \mathbb{T}^n } \mu_L(\rho)|D(u)-D(\bar u)|^2 dxd\tau    \\
& +\frac{\nu}{\epsilon}\iint_{ (0,t) \times \mathbb{T}^n } \lambda_L(\rho)[\dive u - \dive \bar u]^2 dxd\tau  \leq \Psi(0) + \iint_{ (0,t) \times \mathbb{T}^n } \big(|E| +|Q| + |E_2| \big)dxd\tau .
\end{align*}
The terms $Q$ and $E$ are exactly the same of Section \ref{sec:stabconv}, that is
\begin{equation*}
E  = \bar{e} \cdot \frac{\rho}{\bar{\rho}} \left( \frac{m}{\rho} - \frac{\bar{m}}{\bar{\rho}} \right)
\end{equation*}
and
\begin{align*}
Q  = & - \frac{1}{\epsilon} \int \int_{[0,t]\times \mathbb{T}^n} \rho \nabla \bar{u} : [(u-\bar{u}) \otimes (u-\bar{u})]dxd\tau \\
& - \frac{1}{\epsilon} \int \int_{[0,t] \times \mathbb{T}^n} \rho \nabla \bar{u} : [(v-\bar{v}) \otimes (v-\bar{v})]dxd\tau - \frac{1}{\epsilon} \int \int_{[0,t]\times \mathbb{T}^n} p(\rho|\bar{\rho}) \dive \bar{u} dxd\tau  \\
& -  \frac{1}{\epsilon} \int_0^t\int_{\mathbb{T}^n} \rho[(\mu''(\rho)\nabla \rho - \mu''(\bar{\rho})\nabla \bar{\rho})((v-\bar{v})\dive \bar{u} - (u-\bar{u})\dive \bar{v})]dxdt \\ 
&  - \frac{1}{\epsilon} \int_0^t\int_{\mathbb{T}^n} \rho (\mu'(\rho)- \mu'(\bar{\rho}))[(v-\bar{v}) \nabla \dive \bar{u} -  (u-\bar{u})\nabla \dive \bar{v}]dxd\tau, \\
\end{align*}
while the new error term $E_2$ is defined as follows: 
\begin{align*}
E_2 &: =  - \frac{2\nu}{\epsilon}\iint_{ (0,t) \times \mathbb{T}^n }\mu_L(\rho) D(\bar u):D(u-\bar u) dxd\tau
  - \frac{\nu}{\epsilon}\iint_{ (0,t) \times \mathbb{T}^n }\lambda_L(\rho)\dive \bar u(\dive u - \dive \bar u)dxd\tau
 \\
 & = : E_{21} + E_{22}.
\end{align*}
Clearly, the terms $Q$ and $E$ can be bounded  as in Theorem \ref{STABILITY}, namely 
\begin{equation*}
\iint_{[0,t]\times \mathbb{T}^n}|E|dx\tau \leq CT\epsilon^4 + \frac{1}{4\epsilon^2}\iint_{[0,t]\times \mathbb{T}^n} \rho\left|\frac{m}{\rho} - \frac{\bar m }{\bar \rho}\right|^2dxd\tau,
\end{equation*}
where $C$  depends on the bound for $\bar \rho$ and on the (uniform)   $L^1$ bound for $\rho$. Here we also used the fact that $\bar e = O(\epsilon)$.
Moreover, we recall the estimate for $Q$ as well:
 \begin{equation*}
 \iint_{[0,t]\times \mathbb{T}^n}|Q|dxd\tau \leq \frac{1}{4\epsilon^2}\iint_{[0,t]\times \mathbb{T}^n} \rho\left|\frac{m}{\rho} - \frac{\bar m }{\bar \rho}\right|^2dxd\tau + 
  \tilde C \int_{[0,t]}\Psi(\tau)d\tau,
 \end{equation*}
where $\tilde  C$  depends 
  on   the bounds $\dive \bar u/ \epsilon=O(1)$ and $\nabla \dive \bar u/ \epsilon=O(1)$. 
  
  To bound the new terms $E_{21}$ and $E_{22}$ we shall use  the uniform bound \eqref{lame2} a sfollows.
%   let us start by observing that,
% \begin{equation*}
% E_{21} : = \textcolor{red}{- \frac{2 \nu}{\epsilon} \iint_{(0,t) \times \mathbb{T}^n} \mu(\rho) D(\bar u):D(u-\bar u)dxd\tau }
% \end{equation*}
% and
% \begin{equation*}
% E_{22} : = \textcolor{red}{-\frac{\nu}{\epsilon} \iint_{(0,t) \times \mathbb{T}^n} \lambda(\rho)  \dive \bar u(\dive u - \dive \bar u) dx\tau. }
% \end{equation*}
% To this end we use the uniform bound $E_o < \infty$ in the Definition $\ref{deFNS}$, in particular we make use of the relation between the exponents of the pressure $\rho^{\gamma}$ and $\mu(\rho)=\rho^{\frac{s+3}{2}}$, namely $s+2 \leq \gamma$. This fact implies:
%where $\bar C$ is uniformly in $\epsilon$ thanks to $E_o$. 
% Finally, taking into account also the bound $M$ in Definition $\ref{deFNS}$ we get $\mu(\rho) \in L^{\infty}((0,t);L^1(\mathbb{T}^n))$ uniformly in $\epsilon$. 
% Therefore $E_{21}$ is controlled in this way:
% %Then, using $\frac{D(\bar u)}{\epsilon} = O(1)$ and $\mu(\rho) \in L^{\infty}_t L^1_x$ we have
% \begin{align*}
% E_{21} = & \textcolor{red}{ - \frac{2\nu}{\epsilon}\iint_{(0,t) \times \mathbb{T}^n} \mu(\rho)D(\bar u):D(u-\bar u) dxd\tau  \leq }  \\
% & \textcolor{red}{\nu C_2 \epsilon + \frac{\nu}{\epsilon} \iint_{(0,t) \times \mathbb{T}^n} \mu(\rho)|D(u)-D(\bar u)|^2dxd\tau } 
% \end{align*}
% Since:
\begin{align*}
& E_{21} = - \frac{2 \nu}{\epsilon} \iint_{(0,t) \times \mathbb{T}^n} \mu_L(\rho)D(\bar u): [D(u-\bar u)]dxd\tau \leq \\
%& - \sum_{i,j =1}^{n}\iint_{(0,t) \times  \mathbb{T}^n}  \frac{2\sqrt{\nu}}{\sqrt{\epsilon}} \sqrt{\mu(\rho)} \frac{\partial \bar u_i}{\partial x_j}  \frac{ \sqrt{\nu}}{\sqrt{\epsilon}} \sqrt{\mu(\rho)} \frac{\partial( u_i- \bar u_i)}{\partial x_j} dxd\tau \leq \\
& \; \frac{4\nu}{\epsilon} \iint_{(0,t) \times \mathbb{T}^n} \mu_L(\rho)|D(\bar u)|^2dxd\tau + \frac{\nu}{\epsilon}\iint_{(0,t) \times \mathbb{T}^n}\mu_L(\rho)|D(u- \bar u)|^2 dxd\tau \leq \\
& \; \frac{\nu}{\epsilon} \iint_{(0,t) \times \mathbb{T}^n} \mu_L(\rho)|D(u-\bar u )|^2 dxd\tau + \nu C_2T \epsilon,
\end{align*}
where we have used ${D(\bar u)} = O(\epsilon)$, and $C_2$ depends also on $E_o$ in view of \eqref{lamecontrol}. 
The estimate for  $E_{22}$ is analogous:  we use the fact that ${\dive \bar u }= O(\epsilon)$ as follows
\begin{align*}
E_{22} = 
%& \textcolor{red}{- \frac{\nu}{\epsilon} \iint_{(0,t) \times \mathbb{T}^n} \lambda(\rho)  \dive u \dive \bar u -|\dive \bar u|^2  dx\tau }=  \\
&- \frac{\nu}{\epsilon} \iint_{(0,t) \times \mathbb{T}^n} \lambda_L(\rho) \dive \bar u( \dive u -\dive \bar u)  dx\tau  \leq \\
%&\textcolor{red}{-\sum_{i}^{n} \iint_{(0,t) \times \mathbb{T}^n} \frac{\sqrt{\nu}}{ \sqrt{2 \epsilon}} \sqrt{\lambda(\rho)} \frac{\partial (u_i - \bar u_i)}{\partial x_i} \frac{\sqrt{2\nu}}{\sqrt{\epsilon}} \sqrt{\lambda(\rho)} \frac{\partial \bar u_i}{\partial x_i} dxd\tau} \leq \\
&  \frac{\nu}{2\epsilon} \iint_{(0,t) \times \mathbb{T}^n} \lambda_L(\rho) |\dive (u - \bar u)|^2 dxd\tau  + \frac{2\nu}{\epsilon} \iint_{(0,t) \times \mathbb{T}^n} \lambda_L(\rho) |\dive \bar u|^2 dxd\tau \leq \\
& +\nu C_4T \epsilon + \frac{\nu}{2\epsilon} \iint_{(0,t) \times \mathbb{T}^n} \lambda_L(\rho) |\dive (u - \bar u)|^2 dxd\tau,
\end{align*}
where $C_4$ depends also on $E_o$, again using \eqref{lamecontrol}.

Finally we get:
\begin{align*}
& \Psi(t) + \frac{1}{2\epsilon^2} \iint_{[0,t]\times \mathbb{T}^n} \rho \left| \frac{m}{\rho} - \frac{\bar{m}}{\bar{\rho}} \right|^2dxd\tau + \frac{\nu}{\epsilon} \iint_{ (0,t) \times \mathbb{T}^n } \mu_L(\rho)|D(u)-D(\bar u)|^2 dxd\tau    \nonumber\\
&+ \frac{\nu}{2\epsilon}\iint_{ (0,t) \times \mathbb{T}^n } \lambda_L(\rho)[\dive u - \dive \bar u]^2 dxd\tau  \leq \Psi(0) + \tilde{C}\epsilon^4 + \nu \bar{C} \epsilon +  \int_0^t \Psi(\tau) d\tau,
\end{align*}
and, since from the relation $\eqref{lame2}$  we obtain
\begin{align*}
0 &\leq \frac{\nu}{\epsilon} \iint_{ (0,t) \times \mathbb{T}^n } \frac{1}{2}\left( \lambda_L(\rho) + \frac{2}{n}\mu_L(\rho) \right) |\dive(u-\bar{u})|^2dxd\tau  \\
& \leq \frac{\nu}{\epsilon} \iint_{ (0,t) \times \mathbb{T}^n } \mu_L(\rho)|D(u)-D(\bar u)|^2 dxd\tau +  \frac{\nu}{2\epsilon}\iint_{ (0,t) \times \mathbb{T}^n } \lambda_L(\rho)[\dive u - \dive \bar u]^2 dxd\tau,
\end{align*}
 the Gronwall's Lemma gives the result.
\end{proof}
\begin{remark}
Let us emphasize that the choice $\mu_L(\rho) = \mu(\rho)$ and $\lambda_L(\rho) = \lambda(\rho)$ is compatible with \eqref{lame2} and \eqref{lamecontrol} in the range of exponents considered here. 
Indeed, we have  $\mu(\rho)=\rho^{\frac{s+3}{2}}$ with $s\geq -1$, $s+2 \leq \gamma$, and $\gamma>1$, and $\lambda(\rho)=\frac{s+1}{2} \mu(\rho)$. Then  $\mu(\rho)$ and $\lambda(\rho)$ are both nonnegative and
\begin{equation*}
% \left\| \frac{\mu(\rho)}{\sqrt{\rho}} \right\|_{L^{\infty}(0,t);L^2(\mathbb{T}^n)} , 
\left\| \mu_L(\rho) \right\|_{L^{\infty}((0,t);L^1(\mathbb{T}^n))} , \left\| \lambda_L(\rho) \right\|_{L^{\infty}((0,t);L^1(\mathbb{T}^n))} \leq \bar C||\rho||_{L^{\infty}((0,t)L^{\gamma}(\mathbb{T}^n))} \leq  \bar C E_o^{\frac{1}{\gamma}}.
\end{equation*}

Moreover, it is worth to observe the difference between the stability estimate \eqref{stabtheosec4} obtained for the Euler--Korteweg model and \eqref{stabtheosec5} of Theorem \ref{theo:stabNSK}. Besides the common control of the initial relative entropy $\Psi(0)$, the latter gives a control of the errors of the form $O(\epsilon^4) + O(\nu \epsilon)$, which is consistent with the one in  \eqref{stabtheosec4} as $\nu\to 0+$. In other words, the stability estimate obtained in the Euler-Korteweg case is better, nevertheless it is recovered by the one obtained in the presence of the viscosity terms.
The leeway which allows us to perform this estimate in the case of the high friction limit for the Navier--Stokes--Korteweg system is  linked to the fact the viscosity terms appear at an intermediate order in the Hilbert expansion and they are ``less singular'' with respect to the ones coming from the friction term, and therefore they can be controlled in the relative entropy estimate.
\end{remark}

\section*{Acknowledgements}
The authors would like to thank A.E. Tzavaras for many helpful discussions they have about the topics of this paper.

\end{document}